\documentclass[a4paper,11pt, twosides]{amsart}
\pdfoutput=1

\usepackage{mathtools}
\mathtoolsset{showonlyrefs=true}

\usepackage{amsmath,amssymb,amsthm,amsfonts, mathrsfs, fancyhdr}
\usepackage[dvipsnames,dvipsnames,svgnames,table]{xcolor}
\usepackage[expansion=false]{microtype}
\usepackage[all]{xy}

\usepackage[leqno]{amsmath}
\makeatletter
\newcommand{\leqnomode}{\tagsleft@true\let\veqno\@@leqno}
\newcommand{\reqnomode}{\tagsleft@false\let\veqno\@@eqno}
\makeatother

\usepackage{anyfontsize}

\usepackage{upgreek}

\usepackage{enumitem}
\makeatletter
\newcommand{\mylabel}[2]{#2\def\@currentlabel{#2}\label{#1}}
\makeatother

\usepackage{todonotes}

\setlist[description]{leftmargin=*}
\usepackage{amscd}
\usepackage[active]{srcltx}
\usepackage{verbatim}
\usepackage{epsfig,graphicx,color}
\usepackage{graphicx}
\usepackage{mathtools,xparse}
\usepackage[english]{babel}

\usepackage{tabularx,ragged2e}
\newcolumntype{L}{>{\RaggedRight\arraybackslash}X}

\usepackage{xltabular}
\usepackage{booktabs}
\usepackage{caption}

\usepackage{pgf, tikz, tikz-cd}
\usepackage{dsfont}
\usetikzlibrary{matrix}

\usepackage{capt-of}
\usepackage{float}

\usepackage{verbatim}
\usepackage{cite}
\usepackage{manyfoot}

\usepackage{fouriernc}
\usepackage[left=2.7cm,right=2.7cm,top=3cm,bottom=3cm]{geometry}

\usepackage[unicode,bookmarks, pdftex, backref = page]{hyperref}
\definecolor{newblue}{RGB}{0,102,204}
\definecolor{newred}{RGB}{206,32,41}
\hypersetup{colorlinks=true,citecolor=newblue,linkcolor=newred,
  urlcolor=magenta,
pdfpagemode=UseNone, breaklinks=true}

\usepackage{thmtools}

\usepackage{apptools}
\AtAppendix{\counterwithin{theorem}{section}}

\newtheorem{theorem}{Theorem}[section]
\newtheorem{proposition}[theorem]{Proposition}

\newtheorem{corollary}[theorem]{Corollary}

\newtheorem*{A}{Theorem A}
\newtheorem*{B}{Theorem B}
\newtheorem*{C}{Corollary C}

\theoremstyle{definition}
\newtheorem{remark}[theorem]{Remark}

\newtheorem{definition}[theorem]{Definition}

\renewcommand{\setminus}{-}

\newcommand{\mZ}{\mathbb{Z}}

\renewcommand{\to}{\longrightarrow}
\renewcommand{\t}{\mathsf{t}}
\renewcommand{\r}{\mathsf{r}}
\newcommand{\z}{\mathsf{z}}

\renewcommand{\S}{\mathfrak{S}}

\usepackage{epigraph}
\usepackage{listings}

\definecolor{codegreen}{rgb}{0,0.6,0}
\definecolor{codegray}{rgb}{0.5,0.5,0.5}
\definecolor{codepurple}{rgb}{0.58,0,0.82}
\definecolor{backcolour}{rgb}{0.95,0.95,0.92}

\lstdefinestyle{mystyle}{
    backgroundcolor=\color{backcolour},   
    commentstyle=\color{codegreen},
    keywordstyle=\color{magenta},
    numberstyle=\tiny\color{codegray},
    stringstyle=\color{codepurple},
    basicstyle=\ttfamily\footnotesize,
    breakatwhitespace=false,         
    breaklines=true,                 
    captionpos=t,                    
    keepspaces=true,                 
    numbers=left,                    
    numbersep=5pt,                  
    showspaces=false,                
    showstringspaces=false,
    showtabs=false,                  
    tabsize=2
}

\title[Groups of order $64$ and non-homeomorphic double Kodaira
fibrations]{Groups of order $64$ and non-homeomorphic double Kodaira fibrations
with the same biregular invariants}

\date{}
\author[Francesco Polizzi]{Francesco Polizzi $^{*}$}
\address{Dipartimento di Matematica e Applicazioni ``Renato Caccioppoli"
  \newline\indent
  Universit\`a degli Studi di Napoli ``Federico II"
  \newline\indent
Via Cintia, Monte S. Angelo
  \newline\indent
  I-80126  Napoli, Italy}
\email{francesco.polizzi@unina.it}

\author[Pietro Sabatino]{Pietro Sabatino}
\address{Institute for High Performance Computing and Networking (ICAR-CNR)
    \newline\indent
via P. Bucci 8/9C
\newline\indent
87036 Rende (CS), Italy}
\email{pietro.sabatino@icar.cnr.it}



\begin{document}


\begin{abstract}
Let $\Sigma_b$ be a closed Riemann surface of genus $b$. We investigate finite quotients $G$ of the pure braid group on two strands $\mathsf{P}_2(\Sigma_b)$ which do not factor through $\pi_1(\Sigma_b \times \Sigma_b)$. Building on our previous work on some special systems of generators on finite groups that we called \emph{diagonal double Kodaira structures}, we prove that, if $G$ has not order $32$, then $|G| \geq 64$, and we completely classify the cases where equality holds.  
In the last section, as a geometric application of our algebraic results, we construct two $3$-dimensional families of double Kodaira fibrations having the same biregular invariants and the same Betti numbers but different fundamental group. 
\end{abstract}

\maketitle



\setcounter{section}{-1}

\section{Introduction} \label{sec:intro}

A \emph{Kodaira fibration} is a smooth, connected holomorphic fibration $f_1 \colon S \to B_1$, where $S$ is a compact complex surface and $B_1$ is a compact closed curve, which is \emph{non-isotrivial}, meaning that the fibres are not all biholomorphic to one another. The genus of the base curve $B_1$, denoted by $b_1 := g(B_1)$, is called the \emph{base genus}, while the genus of any fibre $F$, denoted $g := g(F)$, is known as the \emph{fibre genus}. A surface $S$ occurring as the total space of a Kodaira fibration is called a \emph{Kodaira fibred surface}. For any Kodaira fibration, the inequalities $b_1 \ge 2$ and $g \geq 3$ are always satisfied, as established in \cite[Theorem 1.1]{Kas68}. Because of the restriction on the base genus, the surface $S$ cannot contain rational or elliptic curves, implying that $S$ is minimal. Furthermore, by the sub-additivity of the Kodaira dimension, $S$ is a surface of general type and is thus algebraic.

Kodaira fibred surfaces raise both fascinating and challenging questions that lie at the intersection of the algebro-geometric properties of compact complex surfaces and the topological aspects of the underlying closed, oriented $4$-manifolds. These surfaces can be investigated not only through traditional algebraic geometry tools but also by applying techniques from geometric topology. Notable methods include the Meyer signature formula, Birman-Hilden relations in the mapping class group, and the decomposition of Lefschetz fibrations, as outlined in \cite{En98, EKKOS02, St02, L17}. For an extensive overview and additional references, we suggest consulting \cite{Cat17}.

The classical example introduced by Kodaira (see \cite[Chapter V, Section 14]{BHPV03}), along with its variants described in \cite{At69, Hir69}, are constructed via ramified coverings of product of curves, thus providing a pair of Kodaira fibrations in the same surface. This leads to the concept of a \emph{double Kodaira fibration}, which is formally defined as follows \cite{Zaal95, LeBrun00, BDS01, BD02, CatRol09, Rol10, LLR17}:

\begin{definition}
  A \emph{double Kodaira surface} is a compact, complex surface $S$, endowed
  with a \emph{double Kodaira fibration}, namely a surjective, holomorphic map
  $f \colon S \to B_1 \times B_2$ yielding, by composition with the natural
  projections, two Kodaira fibrations $f_i \colon S \to B_i$, $i=1, \, 2$.
\end{definition}

In the sequel, we will detail our strategy for constructing double Kodaira fibrations by utilizing the techniques introduced in \cite{CaPol19, PolSab22}, along with presenting our new results. Our primary approach involves the ``detopologization" of the problem, namely, transforming it into a purely algebraic one. This will be applied in the context of \emph{diagonal} double Kodaira fibrations: these are defined as Stein factorizations $f \colon S \to \Sigma_{b_1} \times \Sigma_{b_2}$ of finite Galois covers of the form 
\begin{equation} \label{eq:intro-diagonal}
  \mathbf{f} \colon S \to \Sigma_b \times \Sigma_b,
\end{equation}
branched with order $n \geq 2$ over the diagonal $\Delta \subset \Sigma_b
\times \Sigma_b$, where $\Sigma_b$ stands for a closed Riemann surface of genus $b$. According to the Grauert-Remmert extension theorem and Serre's GAGA principle, the existence of a $G$-cover $\mathbf{f}$ as in \eqref{eq:intro-diagonal} corresponds to the existence of a group epimorphism 
\begin{equation} \label{eq:intr-varphi}
  \varphi \colon \pi_1(\Sigma_b \times \Sigma_b - \Delta) \to G,
\end{equation}
and the condition that $\mathbf{f}$ branches with order $n$ along $\Delta$ can be expressed by requiring that $\varphi(\gamma_{\Delta})$ has order $n$ in $G$, where $\gamma_{\Delta}$ represents a loop in $\Sigma_b \times \Sigma_b - \Delta$ that encircles the diagonal. The condition $n \geq 2$ means  that $\varphi$ does not factor through $\pi_1(\Sigma_b \times \Sigma_b)$ and, crucially, it also implies that $G$ must be non-abelian, since $\gamma_{\Delta}$ is a non-trivial commutator in $\pi_1(\Sigma_b \times \Sigma_b - \Delta)$.  We refer to such epimorphisms, for which $\varphi(\gamma_{\Delta})$ is non-trivial, as \emph{admissible}.

Recall that the fundamental group $\pi_1(\Sigma_b \times \Sigma_b - \Delta)$ is isomorphic to $\mathsf{P}_2(\Sigma_b)$, the pure braid group with two strands on $\Sigma_b$. This group has a finite geometric presentation with $4b+1$ generators, as proved in \cite[Theorem 7]{GG04}. By applying an admissible epimorphism to these generators, we get an ordered sequence
\begin{equation}
\mathfrak{S} = (\r_{11}, \, \t_{11}, \ldots, \r_{1b}, \, \t_{1b}, \, \r_{21}, \, \t_{21}, \ldots, \r_{2b}, \, \t_{2b}, \, \z),
\end{equation}
consisting of $4b+1$ elements that generate $G$. Among these, $\z$ is a commutator satisfying $o(\z) = n$, and the set is constrained by a certain set of relations coming from the relations of $\mathsf{P}_2(\Sigma_b)$.  This configuration is referred to as a \emph{diagonal double Kodaira structure} of type $(b, \, n)$ on the group $G$. Hence, what begins as the geometric problem of constructing an admissible epimorphism onto $G$ is ultimately converted into the combinatorial and algebraic task of finding a diagonal double Kodaira structure of type $(b, \, n)$ on $G$. 

Via the composition with the two natural projections, and up to a simultaneous Stein factorization, the $G$-cover $\mathbf{f} \colon S \to \Sigma_b \times \Sigma_b$ provides a double Kodaira structure on $S$, called a \emph{diagonal double Kodaira fibration}. In the light of the previous considerations, classifying diagonal double Kodaira fibrations is equivalent to describe finite groups which admit a diagonal double Kodaira structure.  

In \cite{PolSab22} we solved this problem for groups of order at most $32$, showing that the only groups in this range having diagonal double Kodaira structures are the two extra-special groups $G(32, \, 49)$ and $G(32, \, 50)$, and we computed the number of structures of type $(2, \, 2)$  in each case (see Proposition \ref{prop:order32}). 

In this work, that is intended as a sequel to \cite{PolSab22}, we analize groups of order at most $64$. It turns out that there are no groups $G$ with order $33 \leq |G| \leq 63$ admitting double Kodaira structures, see Proposition \ref{prop: no_ddks_less_63}. Our first result concerns the complete classification of  the cases where  $|G|=64$, see Theorem \ref{thm:groups_order64_with_structures}:

\begin{A}
Let $G$ be a group of order $64$ admitting a diagonal double Kodaira structure. The following holds:
\begin{itemize}
\item if $G$ is non-monolithic then it is isomorphic to $G(64, \, t)$, with $t \in \{ 199, \, 200, \,  201, \, 264,\,  265 \};$ 
\item if $G$ is monolithic, then it is isomorphic to $G(64, \, t)$, with $t \in \{249, \, 266\}$. 
\end{itemize}
\end{A}

We are also able to provide the number of diagonal double Kodaira structures with $b=2$ in each of these situations. For all of them we have $n=2$, because we always  get $[G, \, G]=\mathbb{Z}_2$. Furthermore, in the non-monolithic cases, all diagonal double Kodaira structures can be derived as lifts from structures on an extra-special quotient of order 32. In contrast, the structures appearing in the monolithic case are genuinely ``new'', making the monolithic cases particularly interesting from this perspective. 

As a geometric application of our algebraic results, we can describe the biregular and topological invariants of all diagonal double Kodaira fibrations associated with these structure, see Theorem \ref{thm:H1 in order 64}:

\begin{B}
Let let $\mathbf{f} \colon S \to \Sigma_2 \times \Sigma_2$ be the $G$-cover associated with a diagonal double Kodaira structure of type $(b,
  \, n)=(2, \, 2)$ on a finite group $G$ of order $64$, and let $f \colon S \to \Sigma_{b_1} \times \Sigma_{b_2}$ be the corresponding  diagonal double Kodaira
  fibration. Then the occurrences for $H_1(S, \, \mathbb{Z})$, $q(S)= \frac{1}{2} \operatorname{rank} \, H_1(S, \, \mathbb{Z})$, $K_S^2$, $c_2(S)$, $\sigma(S)$, $b_1$, $b_2$, $g_1$, $g_2$ are as in the following table. 
  \begin{table}[H]
  \begin{center}
    \begin{tabularx}{0.85\textwidth}{@{}cccccccccc@{}}
\toprule
 $\mathrm{IdSmallGroup}(G)$ & $H_1(S, \, \mathbb{Z})$ & $q(S)$ & $K_S^2$ & $c_2(S)$ & $\sigma(S)$ & $b_1$ & $b_2$ & $g_1$ & $g_2$ \\
\toprule
 $G(64, \, 199)$ & $\mathbb{Z}^8 \oplus (\mathbb{Z}_2)^4$ & $4$ & $736$ & $320$ & $32$ & $2$ & $2$ & $81$ & $81$ \\
$G(64, \, 200)$ & $\mathbb{Z}^8 \oplus (\mathbb{Z}_2)^4$ & $4$ & $736$ & $320$ & $32$ & $2$ & $2$ & $81$ & $81$ \\
 $G(64, \, 201)$ & $\mathbb{Z}^8 \oplus (\mathbb{Z}_2)^4$ & $4$ & $736$ & $320$ & $32$ & $2$ & $2$ & $81$ & $81$ \\
 $G(64, \, 264)$ & $\mathbb{Z}^{12} \oplus (\mathbb{Z}_2)^3$ & $6$ & $736$ & $320$ & $32$ & $3$ & $3$ & $41$ & $41$\\   
 $G(64, \, 265)$ & $\mathbb{Z}^{12} \oplus (\mathbb{Z}_2)^3$ & $6$& $736$ & $320$ & $32$ & $3$ & $3$ & $41$ & $41$\\   
 $G(64, \, 249)$ & $\mathbb{Z}^8 \oplus (\mathbb{Z}_2)^4$ & $4$  & $736$ & $320$ & $32$ & $2$ & $2$ & $81$ & $81$ \\
 $G(64, \, 266)$ & $\mathbb{Z}^{12} \oplus (\mathbb{Z}_2)^3$ & $6$ & $736$ & $320$ & $32$ & $3$ & $3$ & $41$ & $41$\\   
$G(64, \, 266)$  & $\mathbb{Z}^{12} \oplus (\mathbb{Z}_2)^2 \oplus \mathbb{Z}_4$ & $6$ & $736$ & $320$ & $32$ & $3$ & $3$ & $41$ & $41$\\   
\bottomrule
\end{tabularx}
  \end{center}
\end{table}
\end{B}
Thus, in the case $G=(64,\, 266)$ there are two possibilities for the torsion part of $H_1(S, \, \mathbb{Z})$. Moreover, \emph{every} curve of genus $b=2$ can be used as a starting point for our construction, and so such a construction depends on $3$ parameters. This implies the following result, see Corollary \ref{cor:non-homotopically equivalent}:

\begin{C} 
There exists two $3$-dimensional families $\mathcal{F}_1$, $\mathcal{F}_2$ of Kodaira doubly fibred surfaces with $b_1=b_2=3$, $g_1=g_2=41$, $\sigma(S)=32$ such that  
\begin{itemize}
\item the surfaces in $\mathcal{F}_1$ and those in $\mathcal{F}_2$ have the same biregular invariants
\begin{equation}
p_g(S)=93, \; \; q(S)=6, \; \; K_S^2=736
\end{equation}
\item the surfaces in $\mathcal{F}_1$ and those in $\mathcal{F}_2$ have the same Betti numbers
\begin{equation}
\mathsf{b}_0=\mathsf{b}_4=1, \; \; \mathsf{b}_1=\mathsf{b}_3=12, \; \; \mathsf{b}_2=342
\end{equation}
\item the surfaces in $\mathcal{F}_1$ and those in $\mathcal{F}_2$ have different torsion part of $H_1(S, \, \mathbb{Z})$. In particular, they are not homotopically equivalent.
\end{itemize}
\end{C}

As far as we know, this yields the first explicit example of positive-dimensional families of (doubly) fibred Kodaira surfaces having the same base and fibre genera, the same biregular invariants, the same Betti numbers and different fundamental group, see Remark \ref{rem:first occurrence}.

Let us briefly discuss the computational tools we used. We successfully avoided the use of the computer for groups of order less than $64$, with the exception of two groups of order $54$. However, when investigating groups of order 64, the use of the computational algebra system \verb|GAP4| (see \cite{GAP4})  became indispensable, because there are too many of them. So we implemented an instruction that we called \verb|CheckStructures| and that we used to perform some of our calculations.  For the convenience of the reader, we provide all relevant scripts in a repository on GitHub, see \cite{GAP4DDKS}.

This work is organized as follows. Sections \ref{sec:group-prel-I} and \ref{sec:group-prel-II} cover the algebraic preliminaries necessary for the rest of the paper, focusing in particular on CCT groups and monolithic groups. The main point is that CCT groups cannot admit diagonal double Kodaira structures (Remark \ref{rmk:cct-no-prestructure}), whereas all structures on a non-monolithic group are liftings of structures on a proper quotient (Proposition \ref{prop:monolithic-argument}). In our research we can therefore exclude both non-CCT groups and non-monolithic groups without quotients with structures. This rules out all groups of order less than $64$, except for two of them, and all groups of order $64$, except for $68$ of them.

In Section \ref{sec:DDKS} we define diagonal double Kodaira structures and explain their relationship to their geometric topology counterparts, namely admissible epimorphisms from pure surface braid groups to finite groups. We also prove that there are no diagonal double Kodaira structures in groups of order $33 \leq |G| \leq 63$.

Finally, in Section \ref{sec:order_64} we present  Theorem A, whereas Section \ref{sec:first_homology} contains Theorem B and Corollary C.

The paper ends with an appendix, where we collect the presentations for the non-abelian groups of orders 36, 40, 48, 54, 56, and 60 that were used in our calculations.

\bigskip \bigskip

$\mathbf{Notation \; and \; conventions}$.

If $S$ is a smooth projective surface over $\mathbb{C}$, then $c_1(S)$, $c_2(S)$ denote the first and second Chern class of its tangent bundle $T_S$, respectively. Moreover, $K_S$ is the canonical class, $p_g(S)=h^0(S,
\, K_S)$ is the \emph{geometric genus}, $q(S)=h^1(S, \, K_S)$ is the
\emph{irregularity} and $\sigma(S)$ is the $\emph{signature}$.

Throughout the paper we use the following notation for
groups:
\begin{itemize}
  \item $\mZ_n$: cyclic group of order $n$.
  \item $G=N \rtimes Q$: semi-direct product of $N$ and $Q$, namely, split
    extension of $Q$ by $N$, where $N$ is normal in $G$.
  \item $G=N.Q$: non-split extension of $Q$ by $N$.
  \item $\operatorname{Aut}(G)$: the automorphism group of $G$.
  \item $\mathsf{D}_{p, \, q, \, r}=\mathbb{Z}_q \rtimes \mathbb{Z}_p=
    \langle
    x, \, y \; |
    \; x^p=y^q=1, \; xyx^{-1}=y^r \rangle$: split metacyclic
    group of order
    $pq$. The group
    $\mathsf{D}_{2, \, n, \, -1}$ is the dihedral group of
    order $2n$ and will
    be denoted by $\mathsf{D}_{2n}$.
  \item If $n$ is an integer greater or equal to $4$, we denote by
    $\mathsf{QD}_{2^n}$ the quasi-dihedral group of order
    $2^n$, having
    presentation
    \begin{equation}
      \mathsf{QD}_{2^n} := \langle x, \, y \mid
      x^2=y^{2^{n-1}} = 1, \; xyx^{-1}
      = y^{2^{n-2} - 1} \rangle.
    \end{equation}
  \item The generalized quaternion group of order $4n$ is denoted by
    $\mathsf{Q}_{4n}$ and is presented as
    \begin{equation}
      \mathsf{Q}_{4n} = \langle x, \,y, \, z \mid x^n =
      y^2 = z^2 = xyz \rangle.
    \end{equation}
    For $n=2$ we obtain the usual quaternion group $\mathsf{Q}_8$,
    for which we
    adopt the classical presentation
    \begin{equation}
      \mathsf{Q}_{8}=\langle i,\,j,\,k \mid i^2 = j^2 =
      k^2 = ijk\rangle,
    \end{equation}
    denoting by $-1$ the unique element of order $2$.
  \item $\mathsf{S}_n, \;\mathsf{A}_n$: symmetric, alternating group
    on $n$
    symbols. The Klein subgroup of $\mathsf{A}_4$ is denoted by
    $\mathsf{V}_4$. 
   \item $\mathsf{GL}(n, \, \mathbb{F}_q), \, \mathsf{SL}(n, \,
    \mathbb{F}_q), \, \mathsf{Sp}(n, \, \mathbb{F}_q)$: general linear
    group, special linear group and symplectic group of $n
    \times n$ matrices over a field with $q$ elements.
  \item The order of a finite group $G$ is denoted by $|G|$. If $x
    \in G$,
    the order of
    $x$ is denoted by $o(x)$ and its centralizer in $G$
    by $C_G(x)$.
  \item If $x, \, y \in G$, their commutator is defined as
    $[x,\, y]=xyx^{-1}y^{-1}$.
  \item The commutator subgroup of $G$ is denoted by $[G, \, G]$,
    the center
    of $G$ by $Z(G)$.
  \item If $S= \{s_1, \ldots, s_n \} \subset G$, the subgroup generated
    by $S$ is denoted by $\langle S \rangle=\langle s_1,\ldots,
    s_n \rangle$.
  \item $\mathrm{IdSmallGroup}(G)$ indicates the label of the group
    $G$ in
    the  \verb|GAP4| database of small groups. For instance
    $\mathrm{IdSmallGroup}(\mathsf{D}_4)=G(8, \, 3)$ means
    that $\mathsf{D}_4$
    is the third in the
    list of groups of order $8$.
  \item If $N$ is a normal subgroup of $G$ and $g \in G$, we denote by
  $\bar{g}$
    the image of $g$ in the quotient group $G/N$.
\end{itemize}

\section{Group-theoretical preliminaries I: CCT groups} \label{sec:group-prel-I}

\subsection{CCT-groups}
This subsection relies  on \cite[Section 1]{PolSab22}, to which we refer the reader for more details. 
\begin{definition} \label{def:CCT}
  A finite non-abelian group $G$ is said to be a
  \emph{center commutative-transitive group}
  $($or a CCT-group, for short$)$ if commutativity is a transitive
  relation on the set on non-central elements of $G$. In other words, if $x,
  \, y, \, z	\in G \setminus Z(G)$  and $[x, \, y]=[y, \, z]=1$, 
  then $[x, \, z]=1$.
\end{definition}
Other characterizations of CCT-groups are provided in the statement below, 
whose proof is straightforward.

\begin{proposition} \label{prop:CCT}
  For a finite group $G$, the following properties are equivalent.
  \begin{itemize}
    \item[$\boldsymbol{(1)}$] $G$ is a \emph{CCT}-group.
    \item[$\boldsymbol{(2)}$] For every pair $x, \, y $ of
      non-central elements in $G$, the relation $[x, \, y]=1$ implies
      $C_G(x)=C_G(y)$.
    \item[$\boldsymbol{(3)}$] For every non-central element $x
      \in G$, the centralizer $C_G(x)$ is abelian.
  \end{itemize}
\end{proposition}

Moreover, the following holds.
\begin{proposition} \label{prop:small-CCT}
  Let $G$ be a finite non-abelian group.
  \begin{itemize}
    \item[$\boldsymbol{(1)}$] If $|G|$ is the product of at
      most three prime
      factors $($non necessarily distinct$)$, then $G$
      is a \emph{CCT}-group.
    \item[$\boldsymbol{(2)}$] If $|G|=p^4$, with $p$ prime,
      then $G$ is a
      \emph{CCT}-group.
    \item[$\boldsymbol{(3)}$] If $G$ contains an abelian normal
      subgroup of
      prime index, then $G$ is a \emph{CCT}-group.
  \end{itemize}
\end{proposition}

Groups of order at most $32$ that are not CCT-groups were classified in \cite[Section 2]{PolSab22}: they are $\mathsf{S}_4$ and the two extra-special groups of order $2^5=32$, denoted by $\mathsf{H}_5(\mathbb{Z}_2)$ and $\mathsf{G}_5(\mathbb{Z}_2)$. Here we want to extend this classification to groups of order at most $64$. First of all, looking at Proposition \ref{prop:small-CCT}, we get

\begin{proposition} \label{prop:CCT-order-less-64}
Let $G$ be a finite group such that $33 \leq |G| \leq 64$. If $G$ is not a \emph{CCT}-group, then 
\begin{equation}
|G|\in \{ 36, \,40, \, 48, \, 54, \, 56, \, 60, \, 64\}.
\end{equation}
\end{proposition}  
In the remainder of this section we will do a complete analysis of the first six cases; the case where $|G| = 64$ will be considered in Section \ref{sec:order_64}. We will use for each group the presentation given in the Appendix. 

\subsection{The case \texorpdfstring{$|G|=36$}{|G|=36}} \label{subsec:36}

\begin{proposition} \label{prop:36-non-cct}
The only non-abelian group of order $36$ which is not a \emph{CCT}-group is
$G(36, \, 10)$.
\end{proposition}
\begin{proof}
The following groups contain an abelian normal subgroup $N$ of prime index,
so they are CCT-groups (Proposition \ref{prop:small-CCT}):
\begin{table}[H]
  \begin{center}
    \begin{tabularx}{0.75\textwidth}{cccc}
      \toprule
      $\mathrm{IdSmallGroup}(G)$ & $N$ & Structure description of $N$ &
      $[G:N]$ \\
      \toprule
$G(36, \, 1)$ & $\langle x^2, \, y \rangle$ & $\mathbb{Z}_{18}$ & $2$ \\
 $G(36, \, 3)$ & $\langle a, \, b, \, x^3 \rangle$ & $\mathbb{Z}_2 \times
 \mathbb{Z}_6$ & $3$ \\
$G(36, \, 4)$ & $\langle y \rangle$ & $\mathbb{Z}_{18}$ & $2$ \\
 $G(36, \, 6)$ & $\langle x^{-2}, \, y, \, z \rangle$ & $\mathbb{Z}_3 \times
 \mathbb{Z}_6$ & $2$ \\
$G(36, \, 7)$ & $\langle a, \, b, \, x^2 \rangle$ & $\mathbb{Z}_3 \times
\mathbb{Z}_6$ & $2$ \\
$G(36, \, 11)$ & $\mathsf{V}_4 \times \mathbb{Z}_3$ &$\mathbb{Z}_2 \times
\mathbb{Z}_6$ & $3$ \\
$G(36, \, 12)$ & $\langle y \rangle \times \mathbb{Z}_3$ & $\mathbb{Z}_3
\times \mathbb{Z}_6$ & $2$ \\
$G(36, \, 13)$ & $\langle a, \, b, \, c \rangle$ & $\mathbb{Z}_3 \times
\mathbb{Z}_6$ & $2$ \\
\bottomrule
\end{tabularx}
  \end{center}
   \end{table}
Let us now consider the remaining cases. The group $G(36, \, 9)$  has trivial center and five conjugacy
classes of non-trivial elements, whose representatives are
\begin{equation}
x, \, x^2, \, x^3, \, b, \, ab^2.
\end{equation}
The centralizer of $x, \, x^2, \, x^3, \, b$ is isomorphic to $\mathbb{Z}_4$,
whereas the centralizer of $ab^2$ is isomorphic to $\mathbb{Z}_3 \times
\mathbb{Z}_3$. This shows that $G(36, \, 9)$ is a CCT-group (Proposition \ref{prop:CCT}).

Finally, $G(36, \, 10) \simeq \mathsf{S}_3 \times \mathsf{S}_3$
is not a CCT-group: in fact, the centralizer of
the element $((1), \, (1, \, 2) )$ is isomorphic to the non-abelian group
$\mathsf{S}_3 \times \mathbb{Z}_2$.

\end{proof}

\subsection{The case \texorpdfstring{$|G|=40$}{|G|=40}} \label{subsec:40}

\begin{proposition} \label{prop:40-non-cct}
All non-abelian groups of order $40$ are  \emph{CCT}-groups.
\end{proposition}
\begin{proof}
The following groups contain an abelian normal subgroup $N$ of prime index,
so they are  CCT-groups.
\begin{table}[H]
  \begin{center}
    \begin{tabularx}{0.75\textwidth}{@{}cccc@{}}
      \toprule
      $\mathrm{IdSmallGroup}(G)$ & $N$ & Structure description of $N$ &
      $[G:N]$ \\
      \toprule
$G(40, \, 1)$ & $\langle x^{-2}, \, y \rangle$ & $\mathbb{Z}_{20}$ & $2$ \\
$G(40, \, 4)$ & $\langle a, \, i  \rangle$ & $\mathbb{Z}_{20}$ & $2$ \\
$G(40, \, 5)$ & $ \langle y, \, z \rangle$ & $\mathbb{Z}_{20}$ & $2$ \\
$G(40, \, 6)$ & $\langle y \rangle$ & $\mathbb{Z}_{20}$ & $2$ \\
$G(40, \, 7)$ & $\langle x^2, \,  y, \, z \rangle$ & $\mathbb{Z}_2 \times
\mathbb{Z}_{10}$ & $2$ \\
$G(40, \, 8)$ & $\langle xyx^{-1}, \,  y, \, z \rangle$ & $\mathbb{Z}_2
\times \mathbb{Z}_{10}$ & $2$ \\
$G(40, \, 10)$ & $\langle y, \, z \rangle$ &  $\mathbb{Z}_{20}$ & $2$ \\
$G(40, \, 11)$ & $\langle i, \, z \rangle$ & $\mathbb{Z}_4 \times \mathbb{Z}_5$
& $2$ \\
$G(40, \, 13)$ & $\langle y, \, z, \, w \rangle$ & $\mathbb{Z}_2 \times
\mathbb{Z}_{10}$ & $2$ \\
\bottomrule
\end{tabularx}
  \end{center}
   \end{table}

Let us consider now the group $G(40, \, 3)$. Its center is $\langle x^4
\rangle \simeq \mathbb{Z}_2$, and there are eight conjugacy classes of
non-central elements, whose representatives are
\begin{equation}
x, \, x^2, \, x^3, \, x^5, \, x^6, \, x^7, \, y, \, x^4y.
\end{equation}
The centralizer of $x, \, x^2, \, x^3, \, x^5, \, x^6, \, x^7$ is isomorphic
to $\mathbb{Z}_8$, whereas the centralizer of $y, \, x^4y$ is isomorphic to
$\mathbb{Z}_{10}$; thus $G(40, \, 3)$ is a CCT-group.

Finally, let us consider the group $G(40, \, 12)$. Its center is $\langle
z \rangle \simeq \mathbb{Z}_2$, and there are eight conjugacy classes of
non-central elements, whose representatives are
\begin{equation}
x, \, x^2, \, x^3, \, xz, \, x^2z, \, x^3z, \,y, \, yz.
\end{equation}
The centralizer of $x, \, x^2, \, x^3, \, xz, \, x^2z, \, x^3z$ is isomorphic
to $\mathbb{Z}_2 \times \mathbb{Z}_4$, whereas the centralizer of  $y, \,
yz$ is isomorphic to $\mathbb{Z}_{10}$; thus  $G(40, \, 12)$ is a CCT-group.

\end{proof}

\subsection{The case \texorpdfstring{$|G|=48$}{|G|=48}} \label{subsec:48}

\begin{proposition} \label{prop:48-non-cct}
The non-abelian groups of order $48$ which are not	\emph{CCT}-groups
are $G(48, \, t)$ with $t \in \{15, \, 16, \, 17, \, 18, \, 30, \, 38, \,
39, \, 40, \, 41, \, 48\}.$
\end{proposition}
\begin{proof}
The following groups contain an abelian normal subgroup $N$ of prime index,
so they are CCT-groups.

\begin{table}[H]
  \begin{center}
    \begin{tabularx}{0.75\textwidth}{@{}cccc@{}}
      \toprule
      $\mathrm{IdSmallGroup}(G)$ &  $N$ & Structure description of $N$ &
      $[G:N]$ \\
      \toprule
     $G(48, \, 1)$ & $\langle x^{-2}, y \rangle$ & $\mathbb{Z}_{24}$ & $2$ \\
      $G(48, \, 3)$ & $\langle a, \, b \rangle$ & $(\mathbb{Z}_4)^2$ & $3$ \\
     $G(48, \, 4)$ & $\langle (1\, 2\, 3)\rangle \times \langle z \rangle $
     & $\mathbb{Z}_{24}$ & $2$ \\
     $G(48, \, 5)$ & $\langle  y \rangle$ & $\mathbb{Z}_{24}$ & $2$ \\
     $G(48, \, 6)$ & $\langle  y \rangle$ & $\mathbb{Z}_{24}$ & $2$ \\
     $G(48, \, 7)$ & $\langle  y \rangle$ & $\mathbb{Z}_{24}$ & $2$ \\
     $G(48, \, 8)$ & $\langle  a, \, x	\rangle$ &  $\mathbb{Z}_{24}$ & $2$ \\
     $G(48, \, 9)$ & $\langle  x^{-2}, \, y, \, z  \rangle$ & $\mathbb{Z}_2
     \times \mathbb{Z}_{12}$	    & $2$ \\
     $G(48, \, 10)$ & $\langle x^{-2}, \, y, \, z  \rangle$ & $\mathbb{Z}_2
     \times \mathbb{Z}_{12}$  & $2$ \\
      $G(48, \, 11)$ & $\langle  x^{-2}, \, y, \, z  \rangle$ &  $\mathbb{Z}_2
      \times \mathbb{Z}_{12}$ & $2$ \\
      $G(48, \, 12)$ & $\langle  x^{-2}, \, y, \, z  \rangle$ & $\mathbb{Z}_2
      \times \mathbb{Z}_{12}$  & $2$ \\
      $G(48, \, 13)$ & $\langle  x^{-2}, \, y \rangle$ & $\mathbb{Z}_2 \times
      \mathbb{Z}_{12}$	& $2$ \\
     $G(48, \, 14)$ & $\langle	a, \, b \rangle$ &  $\mathbb{Z}_2 \times
     \mathbb{Z}_{12}$ & $2$ \\
     $G(48, \, 19)$ & $\langle	a, \, b, \, x^2  \rangle$ & $(\mathbb{Z}_2)^2
     \times  \mathbb{Z}_6$ & $2$ \\
     $G(48, \, 21)$ & $\langle	a, \, b, \, x^2  \rangle$ & $(\mathbb{Z}_2)^2
     \times \mathbb{Z}_6$ & $2$ \\

    $G(48, \, 22)$ & $\langle  x^2, \, y, \, z	\rangle$ & $\mathbb{Z}_2
    \times \mathbb{Z}_{12}$  & $2$ \\
     $G(48, \, 24)$ & $\langle	y, \, z  \rangle$ & $\mathbb{Z}_{24}$ & $2$ \\
     $G(48, \, 25)$ & $\langle	y, \, z  \rangle$ & $\mathbb{Z}_{24}$ & $2$ \\
     $G(48, \, 26)$ & $\langle	y, \, z  \rangle$ & $\mathbb{Z}_{24}$ & $2$ \\
     $G(48, \, 27)$ & $\langle	x, \, w  \rangle$ & $\mathbb{Z}_{24}$ & $2$ \\
     $G(48, \, 31)$ & $\mathsf{V}_4 \times \mathbb{Z}_4$ & $(\mathbb{Z}_2)^2
     \times \mathbb{Z}_4$ & $3$ \\
     $G(48, \, 34)$ & $\langle x, \, w \rangle$ & $\mathbb{Z}_2 \times
     \mathbb{Z}_{12}$  & $2$ \\
     $G(48, \, 35)$ & $\langle (1\,2\,3), \, z, \, w \rangle$ & $\mathbb{Z}_2
     \times \mathbb{Z}_{12}$  & $2$ \\
     $G(48, \, 36)$ & $\langle y, \, z \rangle$ & $\mathbb{Z}_2 \times
     \mathbb{Z}_{12}$  & $2$ \\
     $G(48, \, 37)$ & $\langle a, \, b \rangle$ & $\mathbb{Z}_2 \times
     \mathbb{Z}_{12}$  & $2$ \\
     $G(48, \, 42)$ & $\langle	a, \, b, \, x^2  \rangle$ & $(\mathbb{Z}_2)^2
     \times \mathbb{Z}_6$ & $2$ \\
  $G(48, \, 43)$ & $\langle  y, \, z, \,w, \, t  \rangle$ & $(\mathbb{Z}_2)^2
  \times \mathbb{Z}_6$ & $2$ \\
   $G(48, \, 45)$ & $\langle  x, \, y^2, \,z   \rangle$ & $(\mathbb{Z}_2)^2
   \times \mathbb{Z}_6$ & $2$ \\
$G(48, \, 46)$ & $\langle  i, \, z   \rangle$ & $\mathbb{Z}_2 \times
\mathbb{Z}_{12}$ & $2$ \\
$G(48, \, 47)$ & $\langle  y, \, z, \, w   \rangle$ & $\mathbb{Z}_2 \times
\mathbb{Z}_{12}$ & $2$ \\
$G(48, \, 49)$ & $\mathsf{V}_4 \times (\mathbb{Z}_2)^2$ & $ (\mathbb{Z}_2)^4$
&  $3$ \\
$G(48, \, 50)$ & $\langle a, \, b, \, c, \, d \rangle$	& $ (\mathbb{Z}_2)^4$
&  $3$ \\
$G(48, \, 51)$ & $\langle (1\, 2\, 3) \rangle \times (\mathbb{Z}_2)^3$ &
$(\mathbb{Z}_2)^3 \times \mathbb{Z}_3$ & $2$  \\
    \bottomrule

\end{tabularx}
  \end{center}
   \end{table}

Let us consider the other cases. The group $G(48, \, 28)$ has center $\langle xyz \rangle \simeq \mathbb{Z}_2$
and six conjugacy classes of non-central elements, whose representatives are
\begin{equation}
xy, \, x^2yx^2, \, xyx^3y, \, x^7y, \, x^4yx^2, \, x^2yx.
\end{equation}
The respective centralizers are
\begin{equation}
\mathbb{Z}_4, \, \mathbb{Z}_6, \, \mathbb{Z}_8, \, \mathbb{Z}_8, \,
\mathbb{Z}_6, \, \mathbb{Z}_8,
\end{equation}
hence $G(48, \, 28)$ is a CCT-group.

The group $G(48, \, 29)$ has center $\langle z^2 \rangle \simeq \mathbb{Z}_2$
and six conjugacy classes of non-central elements, whose representatives are
\begin{equation}
x, \, y, \, z, \, xz, \, yz, \, xzy.
\end{equation}
The respective centralizers are
\begin{equation}
\mathbb{Z}_2 \times \mathbb{Z}_2, \, \mathbb{Z}_6, \, \mathbb{Z}_8, \,
\mathbb{Z}_8, \, \mathbb{Z}_6, \, \mathbb{Z}_8,
\end{equation}
hence $G(48, \, 29)$ is a CCT-group.

The group $G(48, \, 32)$ has center $\langle xyz, \, w \rangle \simeq
\mathbb{Z}_2 \times \mathbb{Z}_2$ and ten conjugacy classes of non-central
elements, whose representatives are
\begin{equation}
xyx, \, xy, \, xyxw, \, xyw, \, x^3y, \, x^4yx, \, x^3yw, \, x^4yxw, \, x^5,
\, x^5w.
\end{equation}
The respective centralizers are
\begin{equation}
\mathbb{Z}_6 \times \mathbb{Z}_2, \,\mathbb{Z}_4 \times \mathbb{Z}_2,
\,\mathbb{Z}_6 \times \mathbb{Z}_2, \,\mathbb{Z}_4 \times \mathbb{Z}_2,
\, \mathbb{Z}_6 \times \mathbb{Z}_2, \,\mathbb{Z}_6 \times \mathbb{Z}_2,
\,\mathbb{Z}_6 \times \mathbb{Z}_2, \,\mathbb{Z}_6 \times \mathbb{Z}_2,
\,\mathbb{Z}_6 \times \mathbb{Z}_2, \,\mathbb{Z}_6 \times \mathbb{Z}_2,
\end{equation}
hence $G(48, \, 32)$ is a CCT-group.

The group $G(48, \, 33)$ has center $\langle w \rangle \simeq \mathbb{Z}_4$
and ten conjugacy classes of non-central elements, whose representatives are
\begin{equation}
xyx, \, xy, \, xyxw, \, xyw, \, x^3y, \, x^4yx, \, x^3yw, \, x^4yxw, \, x^5,
\, x^5w.
\end{equation}
The respective centralizers are
\begin{equation}
\mathbb{Z}_{12},  \,\mathbb{Z}_4 \times \mathbb{Z}_2,  \,\mathbb{Z}_{12},
\,\mathbb{Z}_4 \times \mathbb{Z}_2, \, \mathbb{Z}_{12}, \,\mathbb{Z}_{12},
\, \mathbb{Z}_{12}, \,\mathbb{Z}_{12}, \,\mathbb{Z}_{12}, \,\mathbb{Z}_{12},
\end{equation}
hence $G(48, \, 33)$ is a CCT-group.

Finally, let us show that the remaining non-abelian groups of order $48$
are not CCT-groups. In each case, we exhibit three non-central elements for
which commutativity is not a transitive relation.
\begin{itemize}
\item $G(48, \, 15)$. We have $[x,\,  y^2]=[y, \, y^2]=1$, but $[x, \, y]
\neq 1$.
\item $G(48, \, 16)$. We have $[y,\,  z]=[z, \, xyx]=1$, but $[y, \, x yx]
\neq 1$.
\item $G(48, \, 17)$. We have $[i,\,  a]=[a, \, j ]=1$, but $[i, \, j] \neq 1$.
\item $G(48, \, 18)$. We have $[a,\,  x^2]=[x^2, \, x ]=1$, but $[a, \, x]
\neq 1$.
\item $G(48, \, 30)$. We have $[x,\,  yzy^{-1}]=[yzy^{-1}, \, z ]=1$, but $[x,
\, z] \neq 1$.
\item $G(48, \, 38)$. We have $[(1, \, x), \, ((2 \, 3), \, 1)]=[((2 \, 3),
\, 1), \, (1, \, y) ]=1$, but $[(1, \, x), \, (1, \, y)] \neq 1$.
\item $G(48, \, 39)$. We have $[a,\, b]=[b, \, d ]=1$, but $[a, \, d] \neq 1$.
\item $G(48, \, 40)$. We have $[(1, \, j), \, ((2 \, 3), \, (1))]=[((2 \, 3),
\, (1)), \, (1, \, k) ]=1$, but $[(1, \, j), \, (1, \, k)] \neq 1$.
\item $G(48, \, 41)$. We have $[a,\, a^2]=[a^2, \, b ]=1$, but $[a, \, b]
\neq 1$.
\item $G(48, \, 48)= \mathsf{S}_4 \times \mathbb{Z}_2$. It suffices to recall that 	$\mathsf{S}_4$ is not a
CCT-group: in fact, we have
$[(3 \, 4), \, (1 \, 2)(3 \,4)]=[(1 \, 2)(3 \,4), \, (1 \, 3)(2 \,4)]=1$,
but $[(3 \, 4), \, (1 \, 3)(2 \,4)] \neq 1$.
\end{itemize}
\end{proof}

\subsection{The case \texorpdfstring{$|G|=54$}{|G|=54}} \label{subsec:54}

\begin{proposition} \label{prop:54-non-cct}
The only non-abelian group of order $54$ which are not \emph{CCT}-groups
are $G(54, \, 5)$ and $G(54, \, 6)$.
\end{proposition}
\begin{proof}
The following groups contain an abelian normal subgroup $N$ of prime index,
so they are CCT-groups.

\begin{table}[H]
  \begin{center}
    \begin{tabularx}{0.77\textwidth}{@{}cccc@{}}
      \toprule
      $\mathrm{IdSmallGroup}(G)$ &  $N$ & Structure description of $N$ &
      $[G:N]$ \\
      \toprule
$G(54, \, 1)$ & $\langle  y \rangle$ & $\mathbb{Z}_{27}$ & $2$ \\
$G(54, \, 3)$ & $\langle y \rangle \times \mathbb{Z}_3$ & $\mathbb{Z}_3
\times \mathbb{Z}_9$ & $2$ \\
$G(54, \, 4)$ & $\langle (1\, 2\, 3)\rangle \times \langle z \rangle $ &
$\mathbb{Z}_3 \times \mathbb{Z}_9$ & $2$ \\
$G(54, \, 7)$ & $\langle a, \, b \rangle$ & $\mathbb{Z}_3 \times \mathbb{Z}_9$
& $2$ \\
$G(54, \, 10)$ & $\langle y, \,z, \,w  \rangle$ & $\mathbb{Z}_3 \times
\mathbb{Z}_6$ & $3$ \\
$G(54, \, 11)$ & $\langle y, \,z  \rangle$ & $\mathbb{Z}_{18}$ & $3$ \\
$G(54,\, 12) $ & $\langle (1 \, 2\, 3) \rangle \times (\mathbb{Z}_3)^2$ &
$(\mathbb{Z}_3)^3$ & $2$ \\
$G(54,\, 13) $ & $\langle a, \, b, \, y \rangle$  &  $(\mathbb{Z}_3)^3$ &
$2$ \\
 $G(54,\, 14) $ & $\langle a, \, b, \, c \rangle$  &  $(\mathbb{Z}_3)^3$ &
 $2$ \\
      \bottomrule
      \end{tabularx}
  \end{center}
   \end{table}

Let us consider now the group $G(54, \, 8)$. It has center $\langle y \rangle \simeq \mathbb{Z}_3$
and seven conjugacy classes of non-central elements, whose representatives are
\begin{equation}
w, \, x, \, z, \, y^2w, \, xz, \, yw, \, x^2z.
\end{equation}
The respective centralizers are
\begin{equation}
\mathbb{Z}_6, \, \mathbb{Z}_3 \times \mathbb{Z}_3, \, \mathbb{Z}_3 \times
\mathbb{Z}_3, \, \mathbb{Z}_6, \, \mathbb{Z}_3 \times \mathbb{Z}_3, \,
\mathbb{Z}_6, \, \mathbb{Z}_3 \times \mathbb{Z}_3,
\end{equation}
hence $G(54, \, 8)$ is a CCT-group.

Finally, let us show that the remaining non-abelian groups of order $54$
 are not CCT-groups. In both cases, we exhibit three non-central elements
 for which commutativity is not a transitive relation.
\begin{itemize}
\item $G(54, \, 5)$. We have $[ax^2,\,	b]=[b, \, a]=1$, but $[ax^2, \, a]
\neq 1$.
\item $G(54, \, 6)$. We have $[x, \, x^2]=[x^2, \, y^3]=1$, but $[x, \, y^3]
\neq 1$.
\end{itemize}
\end{proof}

\subsection{The case \texorpdfstring{$|G|=56$}{|G|=56|}} \label{subsec:56}
\begin{proposition} \label{prop:56-non-cct}
All non-abelian groups of order $56$ are  \emph{CCT}-groups.
\end{proposition}
\begin{proof}
In fact, all non-abelian groups of order $56$ contain
an abelian normal subgroup $N$ of prime index, as shown in the table below.
\begin{table}[H]
  \begin{center}
    \begin{tabularx}{0.72\textwidth}{@{}cccc@{}}
      \toprule
      $\mathrm{IdSmallGroup}(G)$ &  $N$ & Structure description of $N$ &
      $[G:N]$ \\
      \toprule
$G(56, \, 1)$ & $\langle x^{-2}, \,  y \rangle$ & $\mathbb{Z}_{28}$ & $2$ \\
$G(56, \, 3)$ & $\langle x^{-2}, \,  y \rangle$ & $\mathbb{Z}_{28}$ & $2$ \\
$G(56, \, 4)$ & $\langle y, \, z \rangle$ & $\mathbb{Z}_{28}$ & $2$ \\
$G(56, \, 5)$ & $\langle y \rangle$ & $\mathbb{Z}_{28}$ & $2$ \\
$G(56, \, 6)$ & $\langle x^{-2}, \,  y \rangle$ & $\mathbb{Z}_2 \times
\mathbb{Z}_{14}$ & $2$ \\
$G(56, \, 7)$ & $\langle a, \, yx, \, y^2 \rangle$ & $\mathbb{Z}_2 \times
\mathbb{Z}_{14}$ & $2$ \\
$G(56, \, 9)$ & $\langle y, \, z \rangle$ & $\mathbb{Z}_{28}$ & $2$ \\
$G(56, \, 10)$ & $\langle i, \, z \rangle$ & $\mathbb{Z}_{28}$ & $2$ \\
$G(56, \, 11)$ & $\langle a, \, b, \, c \rangle$ & $(\mathbb{Z}_2)^3$ & $7$ \\
$G(56, \, 12)$ & $\langle y, \, a, \, b \rangle$ & $\mathbb{Z}_2 \times
\mathbb{Z}_{14}$ & $2$ \\
\bottomrule
       \end{tabularx}
  \end{center}
   \end{table}
\end{proof}

\subsection{The case \texorpdfstring{$|G|=60$}{|G|=60}} \label{subsec:60}

\begin{proposition} \label{prop:60-non-cct}
The only non-abelian group of order $60$ which are not \emph{CCT}-groups
are $G(60, \, 7)$ and $G(60, \, 8)$.
\end{proposition}
\begin{proof}
The following groups contain an abelian normal subgroup $N$ of prime index,
so they are CCT-groups.
\begin{table}[H]
  \begin{center}
    \begin{tabularx}{0.75\textwidth}{@{}cccc@{}}
      \toprule
      $\mathrm{IdSmallGroup}(G)$ &  $N$ & Structure description of $N$ &
      $[G:N]$ \\
      \toprule
$G(60, \, 1)$ & $\langle x^{-2}, \,  y, \, z \rangle$ & $\mathbb{Z}_{30}$
& $2$ \\
$G(60, \, 2)$ & $\langle x^{-2}, \,  y, \, z \rangle$ & $\mathbb{Z}_{30}$
& $2$ \\
$G(60, \, 3)$ & $\langle x^{-2}, \,  y \rangle$ & $\mathbb{Z}_{30}$ & $2$ \\
$G(60, \, 9)$ & $\mathsf{V}_4 \times \mathbb{Z}_5$ &$\mathbb{Z}_2 \times
\mathbb{Z}_{10}$ & $3$ \\
$G(60, \, 10)$ & $\langle  y, \, z \rangle$ & $\mathbb{Z}_{30}$ & $2$ \\
$G(60, \, 11)$ & $\langle (1\, 2\, 3) \rangle \times \mathbb{Z}_{10}$ &
$\mathbb{Z}_{30}$ & $2$ \\
$G(60, \, 12)$ & $\langle  y  \rangle$ & $\mathbb{Z}_{30}$ & $2$ \\
\bottomrule
\end{tabularx}
  \end{center}
   \end{table}

Let us consider the other cases. The group $G(60, \, 5) \simeq \mathsf{A}_5$ has trivial center	and four
conjugacy classes of non-trivial elements, whose representatives are
\begin{equation}
(1\, 2)(3\, 4), \, (1\, 2\, 3), \, (1\, 2\, 3\, 4\, 5),\, (1\, 2\, 3\, 5\, 4).
\end{equation}
The respective centralizers are
\begin{equation}
\mathbb{Z}_2 \times \mathbb{Z}_2, \, \mathbb{Z}_3, \,  \mathbb{Z}_5, \,
\mathbb{Z}_5,
\end{equation}
hence $G(60, \, 5)$ is a CCT-group.

The group $G(60, \, 6)$ has center $\langle z \rangle \simeq \mathbb{Z}_3$
and $12$ conjugacy classes of non-central elements, whose representatives are
\begin{equation}
x, \, x^2, \, y, \, xz, \, x^3, \, x^2z, \, yz, \, xz^2, \, x^3z, \, x^2z^2,
\, yz^2, \, x^3z^2.
\end{equation}
The respective stabilizers are isomorphic to $\mathbb{Z}_{12}$, except in the
cases $y$, $yz$, $yz^2$, where they are isomorphic to  $\mathbb{Z}_{15}$. Thus
$G(60, \, 6)$ is a CCT-group.

Finally, let us show that the two remaining non-abelian groups of order $60$
 are not CCT-groups. In both cases, we exhibit three non-central elements
 for which commutativity is not a transitive relation.
\begin{itemize}
\item $G(60, \, 7)$. We have $[x, \, x^2]=[x^2, \, yx^2y] =1$, but $[x, \,
yx^2y] \neq 1$.
\item $G(60, \, 8)$. We have $[x, \, a]=[a, \, y]=1$, but $[x, \, y] \neq 1$.
\end{itemize}
\end{proof}

\section{Group-theoretical preliminaries II: monolithic groups} \label{sec:group-prel-II}
Given a finite group $G$, we define $\operatorname{mon}(G)$ as 
the intersection of all the non-trivial normal subgroups of $G$.  
For instance, $G$ is simple if and only if $\operatorname{mon}(G)=G$. 
The group $G$ is said to be \emph{monolithic} if  $\operatorname{mon}(G) \neq \{1\}$. 
Equivalently, $G$ is monolithic if it contains precisely one minimal non-trivial, normal subgroup.\footnote{In \cite[Section 4]{PolSab22} we called the intersection of all 
non-trivial normal subgroup of a finite group $G$ the \emph{socle} of $G$, and we denoted it by
$\operatorname{soc}(G)$. However, we later realized that this is  a non-standard terminology: 
indeed, in the literature, $\operatorname{soc}(G)$ generally denotes the subgroup of 
$G$ generated by the minimal, non-trivial normal subgroups, so we prefer to use here the notation
$\operatorname{mon}(G)$. By definition, if $G$ is monolithic then $\operatorname{mon}(G)$ 
is non-trivial and coincides with $\operatorname{soc}(G)$. Instead, if $G$ is non-monolithic, 
then  $\operatorname{mon}(G)=\{1\}$ but $\operatorname{soc}(G)$ is always non-trivial. 
For instance, if $G=G(64,\, 199)$ then $\operatorname{mon}(G)=\{1\}$ and
$\operatorname{soc}(G) \simeq \mathbb{Z}_2 \times \mathbb{Z}_2$.}
For instance, a non-trivial product of the form $G=H \times K$ is never monolithic, 
since $H \times \{1\}$ and $\{1\} \times K$ are two non-trivial normal subgroups
intersecting only at the identity.   

\begin{proposition} \label{prop:Mon-if-center-non-trivial}
If $G$ is monolithic and $Z(G) \neq \{1\}$, then $\operatorname{mon}(G)$ is cyclic of prime order. 
\end{proposition}
\begin{proof}
By definition $\operatorname{mon}(G) \subseteq Z(G)$, hence  $\operatorname{mon}(G)$ is abelian and all its subgroups are normal in $G$. Thus the claim follows by the minimality of  $\operatorname{mon}(G)$. 
\end{proof}

Since a $p$-group has non-trivial center, Proposition \ref{prop:Mon-if-center-non-trivial} immediately implies

\begin{corollary} \label{cor:monolithic-p-group}
If $G$ is a monolithic $p$-group, then $\operatorname{mon}(G) \simeq \mathbb{Z}_p$.
\end{corollary}

We end this section by classifying the non-abelian groups that are both non-CCT and monolithic, up to order $63$. The corresponding analysis for groups of order $64$ is more involved and, due to the large number of such groups, required the systematic use of the computer; we will postpone it to Section \ref{sec:order_64}.

\begin{proposition} \label{prop:non-CCT_non-mono}
Let $G$ be a non-abelian, monolithic group of order $|G| \leq 63$. If $G$ is not a $\operatorname{CCT}$-group, then $G$ is isomorphic to one of the following groups$:$
\begin{equation}
\mathsf{S}_4, \; \mathsf{H}_5(\mathbb{Z}_2), \; \mathsf{G}_5(\mathbb{Z}_2), \; G(54, \, 5), \; G(54, \, 6).
\end{equation}
\end{proposition}
\begin{proof}
Clearly $\operatorname{mon}(\mathsf{S}_4) = \mathsf{V}_4$, whereas for the two extra-special groups of order $32$ we have  $\operatorname{mon}(G) = Z(G) \simeq \mathbb{Z}_2$. We can therefore assume $|G| \geq 33$.  In Section \ref{sec:group-prel-I} we showed that there are $15$ non-abelian, non-CCT groups $G$ of order $33 \leq |G| \leq 63$, namely
\begin{table}[H]
  \begin{center}
    \begin{tabularx}{0.65\linewidth}{@{}ccccc@{}}
$G(36, \, 10),$ & $G(48,\, 15),$ & $G(48,\, 16),$ & $G(48,\, 17),$ & $G(48,\, 18),$ \\
$G(48,\, 30),$ & $G(48,\, 38),$ & $G(48,\, 39),$ & $G(48,\, 40),$ & $G(48,\, 41),$ \\
$G(48,\, 48),$ & $G(54,\, 5),$ & $G(54,\, 6),$ & $G(60,\, 7),$ & $G(60,\, 8)$. 
\end{tabularx}
  \end{center}
\end{table}
Let us now show that all these groups, except $G(54, \, 5)$ and $G(54, \, 6)$, contain a pair $N_1$, $N_2$ of non-trivial normal subgroups such that $N_1 \cap N_2 = \{1\}$. 
\begin{itemize}
\item $G(36,\, 10)$. Take $N_1=\mathsf{S}_3\times \{1\}$ and $N_2 = \{1\} \times \mathsf{S}_3$. 
\item $G(48, \, 15)$. Take $N_1 = \langle y^6 \rangle \simeq \mathbb{Z}_2$ and $N_2 = \langle y^4 \rangle \simeq \mathbb{Z}_3$.
\item $G(48, \, 16)$. Take $N_1 = \langle x^2 \rangle \simeq \mathbb{Z}_2$ and $N_2 = \langle z \rangle \simeq \mathbb{Z}_3$.
\item $G(48, \, 17)$. Take $N_1 = \langle i^2 \rangle \simeq \mathbb{Z}_2$ and $N_2 = \langle a \rangle \simeq \mathbb{Z}_3$.
\item $G(48, \, 18)$. Take $N_1 = \langle y^2 \rangle \simeq \mathbb{Z}_2$ and $N_2 = \langle a \rangle \simeq \mathbb{Z}_3$.
\item $G(48, \, 30)$. Take $N_1 = \langle x^2 \rangle \simeq \mathbb{Z}_2$ and $N_2 = \langle y, \, z \rangle \simeq \mathsf{A}_4$.
\item $G(48,\, 38)$. Take $N_1=\mathsf{S}_3 \times \{1\}$ and $N_2 = \{1\} \times \mathsf{D}_8$. 
\item $G(48, \, 39)$. Take $N_1 = \langle b^2 \rangle \simeq \mathbb{Z}_2$ and $N_2 = \langle d \rangle \simeq \mathbb{Z}_3$.
\item $G(48,\, 40)$. Take $N_1=\mathsf{S}_3 \times \{1\}$ and $N_2 = \{1\} \times \mathsf{Q}_8$. 
\item $G(48, \, 41)$. Take $N_1 = \langle b^2 \rangle \simeq \mathbb{Z}_2$ and $N_2 = \langle a^4 \rangle \simeq \mathbb{Z}_3$.
\item $G(48,\, 48)$. Take $N_1=\mathsf{S}_4 \times \{1\}$ and $N_2 = \{1\} \times \mathbb{Z}_2$. 
\item $G(60, \, 7)$. Take $N_1 = \langle y^5 \rangle \simeq \mathbb{Z}_3$ and $N_2 = \langle y^3 \rangle \simeq \mathbb{Z}_5$.
\item $G(60,\, 8)$. Take $N_1=\mathsf{S}_3 \times \{1\}$ and $N_2 = \{1\} \times \mathsf{D}_{10}$. 
\end{itemize}
It remains to prove that $G(54, \, 5)$ and $G(54, \, 6)$ are monolithic groups. 

The group $G=G(54, 5)$ contains five non-trivial normal subgroups, namely
\begin{equation}
G, \; \langle a, \, b, \, x^2 \rangle, \; \langle a, \, b, \, x^3 \rangle, \; \langle b, \, x^2 \rangle, \; \langle a, \, b \rangle, \; \langle b \rangle.
\end{equation}
So $\operatorname{mon}(G) = \langle b \rangle \simeq \mathbb{Z}_3$.

The group $G=G(54, 6)$ contains five non-trivial normal subgroups, namely
\begin{equation}
G, \; \langle x^2, \, y \rangle, \; \langle x^3, \, y \rangle, \; \langle x^2, \, y^3 \rangle, \; \langle y \rangle, \; \langle y^3 \rangle.
\end{equation}
So $\operatorname{mon}(G) = \langle y^3 \rangle \simeq \mathbb{Z}_3$.
\end{proof}

\section{Diagonal double Kodaira structures and diagonal double Kodaira fibrations}
\label{sec:DDKS}
For more details about the content of this section, we refer the reader to \cite{Pol22} and  \cite{PolSab22}. 

Let $G$ be a finite group and let $b, \, n \geq 2$ be two positive integers.  A \emph{diagonal double Kodaira structure} of type $(b, \, n)$ on $G$
  is an ordered set of	$4b+1$ generators
  \begin{equation} \label{eq:ddks}
    \S = (\r_{11}, \, \t_{11}, \ldots, \r_{1b}, \, \t_{1b}, \;
      \r_{21}, \,
    \t_{21}, \ldots, \r_{2b}, \, \t_{2b}, \; \z ),
  \end{equation}
  with $o(\z)=n$, such that the following relations are satisfied. We use the commutator notation in order to indicate relations of conjugacy type, writing for instance $[x, \, y]=zy^{-1}$ instead of $xyx^{-1} = z$.
  \begin{itemize}
    \item {Surface relations}
      \begin{align} \label{eq:presentation-0}
	& [\r_{1b}^{-1}, \, \t_{1b}^{-1}] \,
	\t_{1b}^{-1} \, [\r_{1 \,b-1}^{-1}, \,
	\t_{1 \,b-1}^{-1}] \, \t_{1\,b-1}^{-1}
	\cdots [\r_{11}^{-1}, \, \t_{11}^{-1}]
	\, \t_{11}^{-1} \, (\t_{11} \, \t_{12}
	\cdots \t_{1b})=\z \\
	& [\r_{21}^{-1}, \, \t_{21}] \, \t_{21} \,
	[\r_{22}^{-1}, \, \t_{22}] \,
	\t_{22}\cdots	[\r_{2b}^{-1}, \, \t_{2b}]
	\, \t_{2b} \, (\t_{2b}^{-1} \,
	\t_{2 \, b-1}^{-1} \cdots \t_{21}^{-1})
	=\z^{-1}
      \end{align}
    \item {Conjugacy action of} $\r_{1j}$
      \begin{align} \label{eq:presentation-1}
	[\r_{1j}, \, \r_{2k}]& =1  &  \mathrm{if}
	\; \; j < k \\
	[\r_{1j}, \, \r_{2j}]& = 1 & \\
	[\r_{1j}, \, \r_{2k}]& =\z^{-1} \, \r_{2k}\,
	\r_{2j}^{-1} \, \z \, \r_{2j}\,
	\r_{2k}^{-1} \; \;&  \mathrm{if} \;  \;
	j > k \\
	& \\
	[\r_{1j}, \, \t_{2k}]& =1 & \mathrm{if}\;
	\; j < k \\
	[\r_{1j}, \, \t_{2j}]& = \z^{-1} & \\
	[\r_{1j}, \, \t_{2k}]& =[\z^{-1}, \, \t_{2k}]
	& \mathrm{if}\;  \; j > k \\
	& \\
	[\r_{1j}, \,\z]& =[\r_{2j}^{-1}, \,\z] &
      \end{align}
    \item {Conjugacy action of} $\t_{1j}$
      \begin{align} \label{eq:presentation-3}
	[\t_{1j}, \, \r_{2k}]& =1 & \mathrm{if}\;
	\; j < k \\
	[\t_{1j}, \, \r_{2j}]& = \t_{2j}^{-1}\,
	\z\, \t_{2j} & \\
	[\t_{1j}, \, \r_{2k}]& =[\t_{2j}^{-1},\,
	\z] \; \; & \mathrm{if}  \;\; j >
	k \\
	& \\
	[\t_{1j}, \, \t_{2k}]& =1 & \mathrm{if}\;
	\; j < k \\
	[\t_{1j}, \, \t_{2j}]& = [\t_{2j}^{-1}, \,
	\z] & \\
	[\t_{1j}, \, \t_{2k}]& =\t_{2j} ^{-1}\,\z\,
	\t_{2j}\, \z^{-1}\, \t_{2k}\,\z\,
	\t_{2j} ^{-1}\,\z^{-1}\, \t_{2j}\,\t_{2k}^{-1}
	\; \;  & \mathrm{if}\;	\;
	j > k \\
	&  \\
	[\t_{1j}, \,\z]& =[\t_{2j}^{-1}, \,\z] &
      \end{align}
  \end{itemize}
Note that  $\mathsf{z}$ is a commutator of order $n \geq 2$ in $G$, hence $G$ is necessarily \emph{non-abelian}. 

For later use, let us write down the special case consisting of a diagonal
double Kodaira structure of
type $(2, \, n)$. It is an ordered set of nine generators of $G$
\begin{equation}
\S=  (\r_{11}, \, \t_{11}, \,  \r_{12}, \, \t_{12}, \, \r_{21}, \,
    \t_{21}, \,
  \r_{22}, \, \t_{22}, \, \z ),
\end{equation}
with $o(\z)=n \geq 2$, subject to the following relations:
\reqnomode
\begin{equation} \label{eq:ddks-genus-2}
  \begin{aligned}
    \mathbf{(S1)} & \,\, [\r_{12}^{-1}, \, \t_{12}^{-1}] \,
    \t_{12}^{-1} \,
    [\r_{11}^{-1}, \, \t_{11}^{-1}] \, \t_{11}^{-1}\, (\t_{11}
    \, \t_{12}) =
    \z & \\ \mathbf{(S2)} & \, \, [\r_{21}^{-1}, \, \t_{21}]
    \; \t_{21} \;
    [\r_{22}^{-1}, \, \t_{22}] \, \t_{22}\, (\t_{22}^{-1} \,
    \t_{21}^{-1})=
    \z^{-1} \\
    & \\
    \mathbf{(R1)} & \, \, [\r_{11}, \, \r_{22}]=1 & \mathbf{(R6)}
    & \, \,
    [\r_{12}, \, \r_{22}]=1 \\
    \mathbf{(R2)} & \, \, [\r_{11}, \, \r_{21}]=1 &
    \mathbf{(R7)} & \, \, [\r_{12}, \, \r_{21}]=
    \z^{-1}\,\r_{21}\,\r_{22}^{-1}\,\z\,\r_{22}\,\r_{21}^{-1} \\
    \mathbf{(R3)} & \, \, [\r_{11}, \, \t_{22}]=1 & \mathbf{(R8)}
    & \, \,
    [\r_{12}, \, \t_{22}]=\z^{-1} \\
    \mathbf{(R4)} & \, \, [\r_{11}, \, \t_{21}]=\z^{-1} &
    \mathbf{(R9)} & \, \,
    [\r_{12}, \, \t_{21}]=[\z^{-1}, \, \t_{21}] \\
    \mathbf{(R5)} & \, \, [\r_{11}, \, \z]=[\r_{21}^{-1}, \,
    \z] & \mathbf{(R10)}
    & \, \, [\r_{12}, \, \z]=[\r_{22}^{-1}, \, \z] \\
    & \\
    \mathbf{(T1)} & \, \, [\t_{11}, \, \r_{22}]=1 & \mathbf{(T6)}
    & \, \,
    [\t_{12}, \, \r_{22}]= \t_{22}^{-1}\, \z \, \t_{22} \\
    \mathbf{(T2)} & \, \, [\t_{11}, \, \r_{21}]=\t_{21}^{-1}\,
    \z \, \t_{21}
    & \mathbf{(T7)} & \, \, [\t_{12}, \, \r_{21}]= [\t_{22}^{-1},
    \, \z] \\
    \mathbf{(T3)} & \, \, [\t_{11}, \, \t_{22}]=1 & \mathbf{(T8)}
    & \, \,
    [\t_{12}, \, \t_{22}]=[\t_{22}^{-1}, \, \z] \\
    \mathbf{(T4)} & \, \, [\t_{11}, \, \t_{21}]=[\t_{21}^{-1},
    \, \z] &
    \mathbf{(T9)} & \, \, [\t_{12}, \, \t_{21}]=
    \t_{22}^{-1}\, \z \,\t_{22} \,\z^{-1} \,\t_{21} \,\z
    \,\t_{22}^{-1}\,\z^{-1}
    \,\t_{22}\,\t_{21}^{-1} \\
    \mathbf{(T5)} & \, \, [\t_{11}, \, \z]=[\t_{21}^{-1}, \,
    \z] & \mathbf{(T10)}
    & \, \, [\t_{12}, \, \z]=[\t_{22}^{-1}, \, \z] \\
  \end{aligned}
\end{equation}
\leqnomode

The definition of diagonal double Kodaira structure can be motivated by means of some concepts in geometric topology. Let $\Sigma_b$ be a closed Riemann surface of genus $b$ and let $\mathsf{P}_2(\Sigma_b)$ be the pure braid group with two strands on $\Sigma_b$, namely, the fundamental group of the ordered configuration space 
$\Sigma_b \times \Sigma_b - \Delta$, where $\Delta \subset \Sigma_b \times \Sigma_b$ is the diagonal. Let $A_{12} \in \mathsf{P}_2(\Sigma_b)$  be the braid corresponding to the homotopy class in $\Sigma_b \times \Sigma_b - \Delta$ of a loop in  $\Sigma_b \times \Sigma_b$ that winds once around the diagonal $\Delta$.   
Therefore we can state the 
\begin{proposition} \label{prop:correspondence_DDKS_covers}
 A finite group $G$ admits a diagonal double Kodaira structure of type $(b, \, n)$ if and only if  there is a surjective group homomorphism 
\begin{equation} \label{eq:varphi}
\varphi \colon \mathsf{P}_2(\Sigma_b)  \to G
\end{equation}
such that $\mathsf{z}=\varphi(A_{12})$ has order $n\geq 2$. Note that the last condition implies that $\varphi$ does not factor through $\pi_1(\Sigma_b \times \Sigma_b)$.  
\end{proposition}
A group  epimorphism $\varphi \colon \mathsf{P}_2(\Sigma_b)  \to G$ as in Proposition \ref{prop:correspondence_DDKS_covers}  will be called \emph{admissible}. If $\varphi$ is  admissible, then the  two subgroups
\begin{equation}
  \begin{split}
    K_1&:=\langle  \r_{11}, \, \t_{11}, \ldots, \r_{1b}, \,
    \t_{1b}, \; \z
    \rangle \\
    K_2&:=\langle  \r_{21}, \, \t_{21}, \ldots, \r_{2b}, \,
    \t_{2b}, \; \z	\rangle
  \end{split}
\end{equation}
are both normal in $G$, and that there are two short exact sequences
\begin{equation} \label{eq:K1K2}
  \begin{split}
    &1 \to K_1 \to G \to Q_2 \to 1 \\
    &1 \to K_2 \to G \to Q_1 \to 1,
  \end{split}
\end{equation}
such that the elements $\r_{21}, \, \t_{21}, \ldots, \r_{2b}, \, \t_{2b}$
yield a complete system of coset representatives for $Q_2$, whereas the
elements
$\r_{11}, \, \t_{11}, \ldots, \r_{1b}, \, \t_{1b}$ yield a complete system
of coset
representatives for $Q_1$.
  A diagonal double Kodaira structure on $G$ is called  \emph{strong}   if $K_1=K_2=G$.
  
 A \emph{prestructure}
  on $G$ is an ordered set of nine elements
  \begin{equation}
    (\r_{11}, \, \t_{11}, \,  \r_{12}, \, \t_{12}, \, \r_{21},
      \, \t_{21}, \,
    \r_{22}, \, \t_{22}, \, \z ),
  \end{equation}
  with $o(\z)=n \geq 2$, subject to the relations $(\mathrm{R1}), \ldots,
  (\mathrm{R10})$, $(\mathrm{T1}), \ldots, (\mathrm{T10})$ in
  \eqref{eq:ddks-genus-2}.

In other words, the nine elements must satisfy all the relations
defining a diagonal double Kodaira structure of type $(2, \, n)$, except the
surface relations. Here we  are \emph{not} requiring that the
elements of the prestructure generate  $G$. Note that if $G$ admits no prestructure, then it admits no diagonal double Kodaira structure of any type. 

\begin{remark} \label{remark:noncentral}
  Let $G$ be a finite group that admits a prestructure. Then $\z$ and
  all its conjugates are non-trivial elements of $G$ and so, from relations
  $\mathrm{(R4)}$,
  $\mathrm{(R8)}$, $\mathrm{(T2)}$, $\mathrm{(T6)}$, it
  follows that $\r_{11}, \, \r_{12}, \, \r_{21}, \, \r_{22}$ and
  $\t_{12}, \, \t_{12},
  \,	\t_{21}, \t_{22}$ are non-central elements of $G$.
\end{remark}

A key point is now provided by the following

\begin{remark}\label{rmk:cct-no-prestructure}
A CCT-group admits no prestructure, see \cite[Proposition 4.4]{PolSab22}. Thus, when looking for diagonal double Kodaira structures, we may restrict ourselves to the groups which are not CCT-groups.
\end{remark}

Let 
\begin{equation}
\S = (\r_{11}, \, \t_{11}, \ldots, \r_{1b}, \, \t_{1b}, \;
      \r_{21}, \,
    \t_{21}, \ldots, \r_{2b}, \, \t_{2b}, \; \z )
\end{equation}
be a diagonal double Kodaira structure of type $(b, \, n)$ on a finite group $G$. If $G$ is non-monolithic, we can find a non-trivial normal subgroup $N$ of $G$ such that $\z \notin N$. Thus, the ordered set
 \begin{equation}
\bar{\S} = (\bar{\r}_{11}, \, \bar{\t}_{11}, \ldots, \bar{\r}_{1b}, \, \bar{\t}_{1b}, \;
      \bar{\r}_{21}, \,
    \bar{\t}_{21}, \ldots, \bar{\r}_{2b}, \, \bar{\t}_{2b}, \; \bar{\z} )
\end{equation}
is a diagonal double Kodaira structure of type $(b, \, \bar{n})$ on the quotient group $G/N$, with $\bar{n}$ dividing $n$, and $\S$ is obtained by lifting $\bar{\S}$ via the quotient map $\pi \colon G \to G/N$. Consequently we get the following result, cf. \cite[Proposition 4.7]{PolSab22}. 

\begin{proposition} \label{prop:monolithic-argument} Let $G$ be a finite group admitting a diagonal double Kodaira structure $\S$. If $G$ is non-monolithic, then there exist a proper quotient $H$ of $G$ and a diagonal double Kodaira structure $\bar{\S}$ on $H$ such that $\S$ is obtained as a lifting of $\bar{\S}$ via the quotient homomorphism $\pi \colon G \to H$.
\end{proposition}

\begin{corollary} \label{cor:monolithic argument}
Assume that $G$ admits a diagonal double Kodaira structure $\S$, whereas no proper quotient of $G$ does. Then $G$ is monolithic and $\z \in \operatorname{mon}(G)$.
\end{corollary}

The following complete description of diagonal double Kodaira structures on groups of order at most $32$ can be found in \cite[Section 4]{PolSab22}.

\begin{proposition} \label{prop:order32} 
Let  $G$ be a finite group.
\begin{itemize}
\item[$\boldsymbol{(1)}$] If $|G|<32$  then $G$ admits no prestructures and, consequently, no diagonal double Kodaira structures.   
  \item[$\boldsymbol{(2)}$] If $|G|=32$ then $G$ admits a diagonal double Kodaira
  structure if and only if $G$ is extra-special. 
  \item[$\boldsymbol{(3)}$] Both extra-special groups of order $32$
      admit $2211840=1152 \cdot
      1920$ distinct
      diagonal double Kodaira structures of type $(b, \, n)=(2, \,
      2)$. Every such a structure is strong, and moreover      
\begin{itemize}    
    \item  if $G=G(32,
      \,49)=\mathsf{H}_5(\mathbb{Z}_2)$,
      these structures form $1920$ orbits under the natural action
      of $\mathrm{Aut}(G);$
    \item  if $G=G(32,
      \,50)=\mathsf{G}_5(\mathbb{Z}_2)$,
      these structures form $1152$ orbits under the natural action
      of $\mathrm{Aut}(G).$
      \end{itemize}
  \end{itemize}
\end{proposition}

We can now exclude all the remaining groups of order less than $64$.

 \begin{proposition} \label{prop: no_ddks_less_63}
No finite group $G$ with  $33 \leq |G| \leq 63$  admits diagonal double Kodaira structures.
\end{proposition}
\begin{proof}
If $33 \leq |G| \leq 63$ then $G$ cannot have quotients of order $32$ and so, by Corollary \ref{cor:monolithic argument} and Proposition \ref{prop:order32}, if $G$ admits a diagonal double Kodaira structure then $G$ is monolithic and $\z \in \operatorname{mon}(G)$. Now, looking at Remark \ref{rmk:cct-no-prestructure} and Proposition \ref{prop:non-CCT_non-mono}, we see that is sufficient to check that
$G(54, \, 5)$ and $G(54, \, 6)$ have no diagonal double Kodaira structures.  This can be done using Part 1 e Part 2 of the \verb|GAP4| script in the GitHub repository  \cite{GAP4DDKS}. Specifically, the script work as follows. First, Part 1
 creates a \verb|GAP4| instruction called
\verb|CheckStructures|, which provides a complete list of double Kodaira
structures of a given finite group
$G$ in the \verb|GAP4| database, modulo the action of $\operatorname{Aut}(G)$. Then,  Part 2 applies the instruction to $G(54, \, 5)$ and $G(54, \, 6)$.
\end{proof}

If $|G|=64$ then the conclusion of Proposition \ref{prop: no_ddks_less_63} does not hold anymore, but we have the following alternatives.

\begin{proposition} \label{prop:order-64}
Let $G$ be a finite group with $|G|=64$ and admitting a diagonal double Kodaira structure $\S$. 
\begin{itemize}
\item[$\boldsymbol{(1)}$] If $G$ is non-monolithic, then there exist an extra-special quotient $H$ of order $32$ and a diagonal double Kodaira structure $\bar{\S}$ on $H$ such that $\S$ is obtained as a lifting of $\bar{\S}$ via the quotient homomorphism $\pi \colon G \to H$.
\item[$\boldsymbol{(2)}$] If we are in situation $\boldsymbol{(1)}$ and moreover $[G, \, G] \simeq \mathbb{Z}_2$, then there exists an extra-special quotient $H$ of order $32$ such that \emph{every} diagonal double Kodaira structure on $G$ is a  lifting of a structure  on $H$. 
\item[$\boldsymbol{(3)}$] If $G$ does not have extra-special quotients of order $32$, then $G$ is monolithic and $\z \in \operatorname{mon}(G) \simeq \mathbb{Z}_2$.
\end{itemize}
\end{proposition}
\begin{proof}
Part $\boldsymbol{(1)}$ follows directly from Propositions \ref{prop:monolithic-argument} and \ref{prop:order32}. Furthermore, if $[G, \, G]$ has order $2$, then its generator provides the element $\z$ for \emph{every} diagonal double Kodaira structure on $G$; taking $H=G/N$, where $N$ is a non-trivial normal subgroup of $G$ such that $\z \notin N$, we get $\boldsymbol{(2)}$. Finally, part  $\boldsymbol{(3)}$ is a consequence of Corollaries \ref{cor:monolithic argument} and \ref{cor:monolithic-p-group}.  
\end{proof}

\begin{remark} \label{rem:H-not-unique}
The ``universal" extra-special subgroup $H$ in part $\boldsymbol{(2)}$ of Proposition \ref{prop:order-64} needs not be unique. In fact, the diagonal double Kodaira structures on $G=G(64, 201)$ can be lifted from both extra-special groups of order $32$, see Subsection \ref{subsec:64_201}.
\end{remark}

We end this section by recalling  how existence of  diagonal double
Kodaira structures is equivalent to the existence of some special
double Kodaira fibrations (see the Introduction for the definition), that we call \emph{diagonal double Kodaira fibrations}. We closely follow the treatment given in \cite[Section 5]{PolSab22}. With a slight abuse of notation, in the sequel we will use the symbol $\Sigma_b$ to indicate both a closed Riemann surface of genus $b$ and its
underlying real surface. By Grauert-Remmert's extension theorem and Serre's
GAGA, the group epimorphism $\varphi \colon \mathsf{P}_2(\Sigma_b) \to G$ 
induced by a diagonal double Kodaira structure  yields the existence of a smooth,
complex, projective surface $S$ endowed with a Galois cover
\begin{equation}
  \mathbf{f} \colon S \to \Sigma_b \times \Sigma_b,
\end{equation}
with Galois group $G$ and branched precisely over $\Delta$ with branching
order $n$, see \cite[Proposition 3.4]{CaPol19}. 

The braid group $\mathsf{P}_2(\Sigma_b)$ is the middle term of two split
short exact sequences
\begin{equation} \label{eq:oggi}
  1\to \pi_1(\Sigma_b \setminus \{p_i\}, \, p_j) \to \mathsf{P}_2(\Sigma_b) \to
  \pi_1(\Sigma_b, \, p_i) \to 1,
\end{equation}
where $\{i, \, j\} = \{1, \, 2\}$, induced by the two natural projections
of pointed topological spaces
\begin{equation} \label{eq:natural-proj}
  (\Sigma_b \times \Sigma_b \setminus \Delta, \,
  (p_1, \, p_2)) \to (\Sigma_b, \, p_i),
\end{equation}
see \cite[Theorem 1]{GG04}. Composing the left
homomorphisms
in \eqref{eq:oggi} with $\varphi \colon \mathsf{P}_2(\Sigma_b) \to G$,
we get two homomorphisms
\begin{equation} \label{eq:varphi-i}
  \varphi_1 \colon \pi_1(\Sigma_b -\{p_2\}, \, p_1) \to G, \quad
  \varphi_2
  \colon \pi_1(\Sigma_b -\{p_1\}, \, p_2) \to G,
\end{equation}
whose respective images coincide with the normal subgroups $K_1$ and $K_2$ defined
in \eqref{eq:K1K2}. By construction, these are
the homomorphisms induced by the restrictions $\mathbf{f}_i \colon \Gamma_i
\to \Sigma_b$  of the Galois cover $\mathbf{f} \colon S \to \Sigma_b
\times \Sigma_b$ to the fibres of the two natural projections $\pi_i
\colon \Sigma_b \times \Sigma_b \to \Sigma_b$. Since $\Delta$ intersects
transversally at a single point all the fibres of the natural projections,
it follows that both such restrictions are branched at precisely one point,
and the number of connected components of the smooth curve $\Gamma_i \subset
S$ equals the index $m_i:=[G : K_i]$ of the normal subgroup $K_i$ in $G$.

So, taking the Stein factorizations of the compositions $\pi_i \circ
\mathbf{f} \colon S \to \Sigma_b$, namely 
\begin{equation} \label{dia:Stein-Kodaira-gi}
  \begin{tikzcd}
    S \ar{rr}{\pi_i \circ \mathbf{f}}  \ar{dr}{f_i} & &
    \Sigma_b	\\
    & \Sigma_{b_i} \ar{ur}{\theta_i} &
  \end{tikzcd}
\end{equation}
where $\theta_i \colon \Sigma_{b_1} \to \Sigma_b$ is \'{e}tale of degree $b_i$, we obtain two distinct Kodaira fibrations $f_i \colon S \to \Sigma_{b_i}$,
hence a double Kodaira fibration by considering the product morphism
\begin{equation}
  f=f_1 \times f_2 \colon S \to \Sigma_{b_1} \times \Sigma_{b_2}.
\end{equation}
\begin{definition} \label{def:diagonal-double-kodaira-fibration}
  We call $f \colon S \to \Sigma_{b_1} \times \Sigma_{b_2}$ the
  \emph{diagonal
  double Kodaira fibration} associated with the diagonal double Kodaira
  structure $\S$ on the finite group $G$. Conversely, we will say that a
  double Kodaira fibration $f \colon S \to \Sigma_{b_1} \times
  \Sigma_{b_2}$
  is \emph{of diagonal type} $(b, \, n)$ if there exists a finite
  group $G$
  and a diagonal double Kodaira structure $\S$ of type $(b, \, n)$
  on it such
  that $f$ is associated with $\S$.
\end{definition}

\medskip

Since the morphism $\theta_i \colon \Sigma_{b_i} \to \Sigma_b$ is \'{e}tale
of degree $m_i$, by using the Hurwitz formula we obtain
\begin{equation} \label{eq:expression-gi}
  b_1 -1 =m_1(b-1), \quad  b_2 -1 =m_2(b-1).
\end{equation}
Moreover, the fibre genera $g_1$, $g_2$ of the Kodaira fibrations $f_1
\colon S \to \Sigma_{b_1}$, $f_2 \colon S \to \Sigma_{b_2}$ are computed
by the formulae
\begin{equation} \label{eq:expression-gFi}
  2g_1-2 = \frac{|G|}{m_1} (2b-2 + \mathfrak{n} ), \quad 2g_2-2 =
  \frac{|G|}{m_2} \left( 2b-2 + \mathfrak{n} \right),
\end{equation}
where $\mathfrak{n}:= 1 - 1/n$. Finally, the surface $S$ fits into a diagram
\begin{equation} \label{dia:degree-f-general}
  \begin{tikzcd}
    S \ar{rr}{\mathbf{f}}  \ar{dr}{f} & & \Sigma_b \times
    \Sigma_b  \\
    & \Sigma_{b_1} \times \Sigma_{b_2} \ar[ur, "\theta_1 \times
    \theta_2"{sloped, anchor=south}] &
  \end{tikzcd}
\end{equation}
so that the diagonal double Kodaira fibration $f \colon S \to  \Sigma_{b_1}
\times \Sigma_{b_2}$ is a finite cover of degree $\frac{|G|}{m_1m_2}$,
branched precisely over the curve
\begin{equation} \label{eq:branching-f}
  (\theta_1 \times \theta_2)^{-1}(\Delta)=\Sigma_{b_1} \times_{\Sigma_b}
  \Sigma_{b_2}.
\end{equation}
Such a curve is always smooth, being the preimage of a smooth divisor via an
\'{e}tale morphism. However, it is reducible in general, see \cite[Proposition
3.11]{CaPol19}. The invariants of $S$ can be now computed as follows,
see \cite[Proposition 3.8]{CaPol19}.

\begin{proposition} \label{prop:invariant-S-G}
  Let $f \colon S \to \Sigma_{b_1} \times \Sigma_{b_2}$ be a diagonal
  double
  Kodaira fibration, associated with a diagonal double Kodaira
  structure $\S$
  of type $(b, \, n)$ on a finite group $G$. Then we have
  \begin{equation} \label{eq:invariants-S-G}
    \begin{split}
      c_1^2(S) & = |G|\,(2b-2) ( 4b-4 + 4 \mathfrak{n}
      - \mathfrak{n}^2 ) \\
      c_2(S) & =   |G|\,(2b-2) (2b-2 + \mathfrak{n})
    \end{split}
  \end{equation}
  where $\mathfrak{n}=1-1/n$.	 Therefore the signature
  of $S$ can be
  expressed as
  \begin{equation} \label{eq:slope-signature-S-G}
    \sigma(S) = \frac{1}{3}\left(c_1^2(S) - 2 c_2(S)
      \right)
      =\frac{1}{3}\,|G|\,(2b-2)\left(1-\frac{1}{n^2}\right).
\end{equation}
\end{proposition}

\begin{remark} \label{rmk:sdks-characterization}
  By definition, the diagonal double Kodaira structure $\S$ is strong
  if and only if $m_1=m_2=1$, that in turn implies $b_1=b_2=b$,
  i.e.,
  $f=\mathbf{f}$. In other words, $\S$ is strong if and only
  if no Stein
  factorization as in \eqref{dia:Stein-Kodaira-gi} is needed or,
  equivalently,
  if and only if the Galois cover $\mathbf{f}\colon S \to \Sigma_b
  \times
  \Sigma_b$ induced by the admissible epimorphism $\varphi \colon \mathsf{P}_2(\Sigma_b)  \to G$ is already a double Kodaira fibration, branched on the diagonal $\Delta \subset \Sigma_b \times \Sigma_b$.
\end{remark}

\begin{remark} \label{rmk:factorization_of_the_cover}
Proposition \ref{prop:monolithic-argument} can be rephrased in geometric terms as follows: if $G$ is non-monolithic, there exists a quotient $H$ of $G$ and a factorization of the $G$-cover $\mathbf{f}$ of the form
\begin{equation} 
 \begin{tikzcd}
    S \ar{rr}{\mathbf{f}}  \ar{dr}{ } & & \Sigma_b \times
    \Sigma_b  \\
    & \bar{S} \ar[ur, "\bar{\mathbf{f}}"{sloped, anchor=south}] &
  \end{tikzcd}
\end{equation} 
  where $\bar{\mathbf{f}} \colon \bar{S} \to \Sigma_b \times
    \Sigma_b$ is a $H$-cover branched on the diagonal $\Delta \subset \Sigma_b \times \Sigma_b$ with branching order $\bar{n} \geq 2$, where $\bar{n}$ divides $n$.
\end{remark}

\section{Diagonal double Kodaira structures on groups of order
\texorpdfstring{$64$}{64}} \label{sec:order_64}

All the following claims can be verified by using Part 3 of the \verb|GAP4| script in the repository  \cite{GAP4DDKS}. There exist precisely $124$ non-abelian, non-CCT groups: $37$ are monolithic 
and $87$ are non-monolithic. By Proposition \ref{prop:order-64}, non-monolithic
groups of order $64$ which do not have extra-special quotients of order $32$ 
cannot admit prestructures, so they can be discarded. There are $56$ of them, so we are left with $37$ 
monolithic groups and $31$ non-monolithic groups, which are listed in the 
following
\begin{proposition} \label{prop:remaining-order-64} 
Let $G$ be a non-abelian, non-\emph{CCT} group of order $64$. 
\begin{itemize}
\item[$\boldsymbol{(1)}$] If $G$ is monolithic, then $G$ is of the form $G(64, \, t)$ with $t$ in the set
\begin{equation} \label{eq:case_64_monolithic}
\begin{split}
\{ & 18, \, 19,\, 25,\, 28,\, 30,\, 32, \,33, \,34,\, 35, \,36,\, 37,\, 41,\, 42,\, 43,\, 46,
\, 91,\, 94, \,102,\, 111, \,125,\, 134,\, \\
   & 135,\, 136,\, 137,\, 138,\, 139, \,152, \,153,\, 154,\, 190,\, 191, \, 249, \,256, \, 257, \, 258,\, 259,\, 266\}.
\end{split}
\end{equation}
\item[$\boldsymbol{(2)}$] If  $G$ is non-monolithic and admits an extra-special quotient of order $32$, then $G$ is of the form $G(64, \, t)$ with $t$ in the set
\begin{equation} \label{eq:case_64_non_monolithic}
\begin{split}
\{ & 199, \,200,\, 201,\, 215, \, 216,\, 217,\, 218,\, 219, \,220,\, 221,\, 222,\, 223,\, 224,\, 225,\, 226,\, 227,\, 228, \,229, \\
& 230,\, 231, \, 232, \, 233, \, 234, \, 235, \, 236, \, 237, \, 238,\, 239, \, 240, \, 264, \, 265\}.
\end{split}
\end{equation}
\end{itemize}
\end{proposition}
Although we considerably reduced the number of cases to check, these groups are still too many to be attacked  one by one by hand. So
we analyzed them by using  our \verb|GAP4| instruction \verb|CheckStructures|, see Part 4 of the \verb|GAP4| script in the repository  \cite{GAP4DDKS}. In this way we obtain
\begin{theorem} \label{thm:groups_order64_with_structures}
Let $G$ be a group of order $64$ admitting a diagonal double Kodaira structure.
\begin{itemize}
\item If $G$ is non-monolithic then it is isomorphic to $G(64, \, t)$, with $t \in \{ 199, 200, 201, 264, 265 \}$. 
\item If $G$ is monolithic, then it is isomorphic to $G(64, \, t)$, with $t \in \{249, 266\}$. 
\end{itemize}
Moreover,  in each case  the number of diagonal double Kodaira structures with $b=2$ is given in the table below. For all these groups we have $[G, \, G] = \mathbb{Z}_2$, so $n=o(\mathsf{z})=2$. Moreover $($and remarkably$)$ the pair $(|K_1|, \, |K_2|)$ only depends on the group, and not on the structure.
\begin{table}[H]
  \begin{center}
    \begin{tabularx}{0.97\textwidth}{@{}ccccc@{}}
\toprule
$\mathrm{IdSmallGroup}(G)$ & $\emph{Monolithic}$ & $\emph{Number of structures}$ & $\emph{Number of structures}$  & $(\, |K_1|, \, |K_2|\,)$ \\
& & \emph{with}\; $b=2$  & \emph{with} \; $b=2$& \\     
& & &  \emph{modulo}\; $\operatorname{Aut}(G)$ \\ 
\toprule
 $G(64, \, 199)$ & $\emph{no}$ & $138240 \cdot 4096$ & $138240$ & $(64, \, 64)$\\   
 $G(64, \, 200)$ & $\emph{no}$ & $46080 \cdot 12288$ & $46080$ & $(64, \, 64)$\\
 $G(64, \, 201)$ & $\emph{no}$ &  $184320 \cdot 3072$ & $184320$ & $(64, \, 64)$ \\
 $G(64, \, 264)$ & $\emph{no}$ & $14400 \cdot 36864$ & $14400$ & $(32, \, 32)$ \\
 $G(64, \, 265)$ & $\emph{no}$ & $8640 \cdot 61440$ & $8640$ & $(32, \, 32)$ \\
 $G(64, \, 249)$ & $\emph{yes}$ &  $368640 \cdot 1536$ & $368640$ & $(64, \, 64)$\\
 $G(64, \, 266)$ & $\emph{yes}$ & $23040 \cdot 23040$ & $23040$ & $(32, \, 32)$ \\  
     \bottomrule
\end{tabularx}
  \end{center}
\end{table}
\end{theorem}

We will now investigate in detail each of these groups separately. Part 6 of the \verb|GAP4| script in  \cite{GAP4DDKS} analyzes the group $G(64, \, 199)$; the scripts analyzing the other groups are similar. In order to check the liftings of the structures from a quotient, we need to apply our instruction \verb|CheckStructures| to the two extra-special groups of order $32$ and to generate the corresponding list of diagonal double Kodaira structures in each case; this is done by using Part 5 of the aforementioned \verb|GAP4|  script.

In the remainder of this section, for the sake of brevity we will symply write ``structure'' instead of ``diagonal double Kodaira structure of type $(2, \, 2)$''.

\subsection{The case $G=(64, 199)$} \label{subsec:64_199}

The group $G=G(64, \, 199)$ is generated by six elements $x_1, \ldots, x_6$ subject to the following set of relations:
\begin{table}[H]
  \begin{center}
    \begin{tabularx}{0.8\textwidth}{@{}llllll@{}}
$x_1^2=x_6,$ &                       &                      &                     &                    &                    \\
$ x_2^2=1,$    & $[x_1, \, x_2]=x_5,$  &                      &                     &                    &                    \\
$x_3^2=1,$   & $[x_1, \, x_3]=1,$    & $[x_2, \, x_3]=x_5,$ &                     &                    &                    \\
$ x_4^2=1,$  & $[x_1, \, x_4]=x_5, $ & $[x_2, \, x_4]=1, $  & $[x_3, \, x_4]=1, $ &                    &                    \\
$x_5^2=1,$   & $[x_1, \, x_5]=1,$    & $[x_2, \, x_5]=1,$   & $[x_3, \, x_5]=1,$  & $[x_4, \, x_5]=1,$ &                    \\
$x_6^2=1,$   & $[x_1, \, x_6]=1,$    & $[x_2, \, x_6]=1,$   & $[x_3, \, x_6]=1,$  & $[x_4, \, x_6]=1,$ & $[x_5, \, x_6]=1.$
\end{tabularx}
  \end{center}
\end{table}

We have $Z(G) = \langle x_5, \, x_6 \rangle \simeq \mathbb{Z}_2 \times \mathbb{Z}_2$ and $[G, \, G]= \langle x_5 \rangle \simeq \mathbb{Z}_2$. Setting $x_6=1$ in the presentation above we get a presentation for the extra-special group $G(32, \, 49)$, so $G$ is a central extension
\begin{equation} \label{eq:central-ext-64_199}
1 \to  \langle x_6 \rangle \to  G \to G(32, \, 49) \to 1.  
\end{equation} 
The exact sequence \eqref{eq:central-ext-64_199} is non-split, otherwise $G$ would be isomorphic to $G(32, \, 49)  \times \mathbb{Z}_2$, a contradiction since the latter group is $G(64, 264)$.  One easily checks that $G$ is non-monolithic: in fact,  $\langle x_5 \rangle$ and $\langle x_6 \rangle$ are two central, and hence normal, subgroups isomorphic to $\mathbb{Z}_2$ and having trivial intersection. Thus, by Proposition \ref{prop:order-64}, every structure on $G$ is obtained by lifting a structure from the extra-special quotient $G(32, \, 49)$. 

All the liftings of a given structure 
 \begin{equation} \label{eq:ddks_199}
    \S = (\r_{11}, \, \t_{11}, \, \r_{12}, \, \t_{12}, \;
      \r_{21}, \,
    \t_{21}, \, \r_{22}, \, \t_{22}, \; \z ),
  \end{equation}
 in $G(32, \, 49)$ are obtained multiplying by $x_6$ any number of elements of the structure among the $\mathsf{r}_{ij}, \, \mathsf{t}_{ij}$. Multiplying to the right or to the left is irrelevant, since $x_6$ is central in $G$; so we have $2^8=256$ possible liftings, and one verifies that all of these liftings provide strong structures on $G$. 

The number of structures in $G(32, \, 49)$ being $1152 \cdot 1920$, we infer that the total number of structures on $G$ is $1152 \cdot 1920 \cdot 256$. 

Finally, since $|\operatorname{Aut}(G)|=4096$, the number of structures on $G$ modulo automorphisms is 
\begin{equation}
(1152 \cdot 1920 \cdot 256)/4096 = 138240. 
\end{equation}

\subsection{The case $G=(64, 200)$} \label{subsec:64_200}

The group $G=G(64, \, 200)$ is generated by six elements $x_1,  \ldots, x_6$ subject to the following set of relations:
\begin{table}[H]
  \begin{center}
    \begin{tabularx}{0.8\textwidth}{@{}llllll@{}}
$x_1^2=x_6,$  &                       &                      &                     &                    &                    \\
$ x_2^2=x_5,$ & $[x_1, \, x_2]=x_5,$  &                      &                     &                    &                    \\
$x_3^2=x_5,$  & $[x_1, \, x_3]=1,$    & $[x_2, \, x_3]=x_5,$ &                     &                    &                    \\
$ x_4^2=1,$   & $[x_1, \, x_4]=x_5, $ & $[x_2, \, x_4]=1, $  & $[x_3, \, x_4]=1, $ &                    &                    \\
$x_5^2=1,$    & $[x_1, \, x_5]=1,$    & $[x_2, \, x_5]=1,$   & $[x_3, \, x_5]=1,$  & $[x_4, \, x_5]=1,$ &                    \\
$x_6^2=1,$    & $[x_1, \, x_6]=1,$    & $[x_2, \, x_6]=1,$   & $[x_3, \, x_6]=1,$  & $[x_4, \, x_6]=1,$ & $[x_5, \, x_6]=1.$
\end{tabularx}
  \end{center}
\end{table}

We have $Z(G) = \langle x_5, \, x_6 \rangle \simeq \mathbb{Z}_2 \times \mathbb{Z}_2$ and $[G, \, G]= \langle x_5 \rangle \simeq \mathbb{Z}_2$. Setting $x_6=1$ in the presentation above we get a presentation for the extra-special group $G(32, \, 50)$, so $G$ is a central extension
\begin{equation} \label{eq:central-ext-64_200}
1 \to  \langle x_6 \rangle \to  G \to G(32, \, 50) \to 1.  
\end{equation} 
The exact sequence \eqref{eq:central-ext-64_200} is non-split, otherwise $G$ would be isomorphic to $G(32, \, 50)  \times \mathbb{Z}_2$, a contradiction since the latter group is $G(64, 265)$.  One easily checks that $G$ is non-monolithic: in fact,  $\langle x_5 \rangle$ and $\langle x_6 \rangle$ are two central, and hence normal, subgroups  isomorphic to $\mathbb{Z}_2$ and having trivial intersection. Thus,  by Proposition \ref{prop:order-64},  every structure on $G$ is obtained by lifting a structure from the extra-special quotient $G(32, \, 50)$. 

Since $x_6$ is central in $G$, all the liftings of a given structure in $G(32, \, 50)$ are obtained by multiplying by $x_6$ any number of elements among the $\mathsf{r}_{ij}, \, \mathsf{t}_{ij}$. As in the previous case  we have $2^8=256$ possible liftings, and one verifies that all of these liftings provide strong structures on $G$. 

The number structures in $G(32, \, 50)$ being $1152 \cdot 1920$, we infer that the total number structures on $G$ is   $1152 \cdot 1920 \cdot 256$

Finally, since $|\operatorname{Aut}(G)|=12288$, the number of structures on $G$ modulo automorphisms is 
\begin{equation}
(1152 \cdot 1920 \cdot 256)/12288 = 46080. 
\end{equation}

\subsection{The case $G=(64, 201)$} \label{subsec:64_201}

The group $G=G(64, \, 201)$ is generated by six elements $x_1,  \ldots, x_6$ subject to the following set of relations:
\begin{table}[H]
  \begin{center}
    \begin{tabularx}{0.8\textwidth}{@{}llllll@{}}
$x_1^2=x_6,$  &                       &                      &                     &                    &                    \\
$ x_2^2=1,$   & $[x_1, \, x_2]=x_5,$  &                      &                     &                    &                    \\
$x_3^2=1,$  & $[x_1, \, x_3]=1,$    & $[x_2, \, x_3]=x_5,$ &                     &                    &                    \\
$ x_4^2=x_5,$ & $[x_1, \, x_4]=x_5, $ & $[x_2, \, x_4]=1, $  & $[x_3, \, x_4]=1, $ &                    &                    \\
$x_5^2=1,$  & $[x_1, \, x_5]=1,$    & $[x_2, \, x_5]=1,$   & $[x_3, \, x_5]=1,$  & $[x_4, \, x_5]=1,$ &                    \\
$x_6^2=1,$  & $[x_1, \, x_6]=1,$    & $[x_2, \, x_6]=1,$   & $[x_3, \, x_6]=1,$  & $[x_4, \, x_6]=1,$ & $[x_5, \, x_6]=1.$
\end{tabularx}
  \end{center}
\end{table}

We have $Z(G) = \langle x_5, \, x_6 \rangle \simeq \mathbb{Z}_2 \times \mathbb{Z}_2$ and $[G, \, G]= \langle x_5 \rangle \simeq \mathbb{Z}_2$. In this case $G$ is a (non-trivial) central extension of both the extra-special groups of order $32$, in fact we have two short exact sequences
\begin{equation} \label{eq:central-ext-64_201}
\begin{split}
& 1 \to  \langle x_6 \rangle \to  G \to G(32, \, 49) \to 1 \\
& 1 \to  \langle x_5x_6 \rangle \to  G \to G(32, \, 50) \to 1.
\end{split}
\end{equation} 
 One easily checks that $G$ is non-monolithic: in fact,  $\langle x_5 \rangle$ and $\langle x_6 \rangle$ are two central, and hence normal, subgroups  isomorphic to $\mathbb{Z}_2$ and having trivial intersection. Thus, by Proposition \ref{prop:order-64}, every structure on $G$ is obtained by lifting a structure either from the extra-special quotient $G(32, \, 49)$ or from the extra-special quotient $G(32, \, 50)$. 

The number of structures on both $G(32, \, 49)$ and  $G(32, \, 50)$ being $1152 \cdot 1920$, we infer as before that the total number of liftings on $G$ is $1152 \cdot 1920 \cdot 256$; furthermore, one verifies  that all these liftings provide strong structures on $G$. 

Finally, since $|\operatorname{Aut}(G)|=3072$, the number of structures on $G$ modulo automorphisms is 
\begin{equation}
(1152 \cdot 1920 \cdot 256)/3072 = 184320. 
\end{equation}

\subsection{The case $G=(64, 264)$} \label{subsec:64_264}

The group $G=G(64, \, 264)$ is generated by six elements $x_1,  \ldots, x_6$ subject to the following set of relations:
\begin{table}[H]
  \begin{center}
    \begin{tabularx}{0.8\textwidth}{@{}llllll@{}}
$x_1^2=1,$  &                       &                      &                     &                    &                    \\
$ x_2^2=1,$   & $[x_1, \, x_2]=x_5,$  &                      &                     &                    &                    \\
$x_3^2=1,$  & $[x_1, \, x_3]=1,$    & $[x_2, \, x_3]=x_5,$ &                     &                    &                    \\
$ x_4^2=1,$ & $[x_1, \, x_4]=x_5, $ & $[x_2, \, x_4]=1, $  & $[x_3, \, x_4]=1, $ &                    &                    \\
$x_5^2=1,$  & $[x_1, \, x_5]=1,$    & $[x_2, \, x_5]=1,$   & $[x_3, \, x_5]=1,$  & $[x_4, \, x_5]=1,$ &                    \\
$x_6^2=1,$  & $[x_1, \, x_6]=1,$    & $[x_2, \, x_6]=1,$   & $[x_3, \, x_6]=1,$  & $[x_4, \, x_6]=1,$ & $[x_5, \, x_6]=1.$
\end{tabularx}
  \end{center}
\end{table}

We have $Z(G) = \langle x_5, \, x_6 \rangle \simeq \mathbb{Z}_2 \times \mathbb{Z}_2$ and $[G, \, G]= \langle x_5 \rangle \simeq \mathbb{Z}_2$. The group $G$ is a central extension of  the extra-special group $G(32, \, 49)$, given by a short exact sequence
\begin{equation} \label{eq:central-ext-64_264}
1 \to  \langle x_6 \rangle \to  G \to G(32, \, 49) \to 1.
\end{equation} 
One checks that such a sequence is split, hence $G \simeq G(32, \, 49) \times \mathbb{Z}_2$ and so $G$ is non-monolithic. Thus, by Proposition \ref{prop:order-64},  every  structure on $G$ is obtained by lifting a structure from the extra-special quotient $G(32, \, 49)$. 

As before, the total number of liftings  on $G$ is $1152 \cdot 1920 \cdot 256$. However, in this case not all liftings give structures: in fact,  some of them do not generate $G$, but only a subgroup of order $32$ inside it. More precisely, one verifies that for each structure on $G(32, \, 49)$ there are precisely $16$ liftings that do not generate $G$, so the correct  number  of structures on $G$ is $1152 \cdot 1920 \cdot (256-16)$. Moreover, all these structures are non-strong: more precisely, in all cases we have $|K_1|=|K_2|=32$.

Finally, since $|\operatorname{Aut}(G)|=36864$, the number of structures on $G$ modulo automorphisms is 
\begin{equation}
(1152 \cdot 1920 \cdot 240)/36864 = 14400. 
\end{equation}
 
\subsection{The case $G=(64, 265)$}

The group $G=G(64, \, 265)$ is generated by six elements $x_1,  \ldots, x_6$ subject to the following set of relations:
\begin{table}[H]
  \begin{center}
    \begin{tabularx}{0.8\textwidth}{@{}llllll@{}}
$x_1^2=1,$    &                       &                      &                     &                    &                    \\
$ x_2^2=x_5,$ & $[x_1, \, x_2]=x_5,$  &                      &                     &                    &                    \\
$x_3^2=x_5,$  & $[x_1, \, x_3]=1,$    & $[x_2, \, x_3]=x_5,$ &                     &                    &                    \\
$ x_4^2=1,$   & $[x_1, \, x_4]=x_5, $ & $[x_2, \, x_4]=1, $  & $[x_3, \, x_4]=1, $ &                    &                    \\
$x_5^2=1,$    & $[x_1, \, x_5]=1,$    & $[x_2, \, x_5]=1,$   & $[x_3, \, x_5]=1,$  & $[x_4, \, x_5]=1,$ &                    \\
$x_6^2=1,$    & $[x_1, \, x_6]=1,$    & $[x_2, \, x_6]=1,$   & $[x_3, \, x_6]=1,$  & $[x_4, \, x_6]=1,$ & $[x_5, \, x_6]=1.$
\end{tabularx}
  \end{center}
\end{table}

We have $Z(G) = \langle x_5, \, x_6 \rangle \simeq \mathbb{Z}_2 \times \mathbb{Z}_2$ and $[G, \, G]= \langle x_5 \rangle \simeq \mathbb{Z}_2$. The group $G$ is a central extension of  the extra-special group $G(32, \, 49)$, given by a short exact sequence
\begin{equation} \label{eq:central-ext-64_265}
1 \to  \langle x_6 \rangle \to  G \to G(32, \, 50) \to 1.
\end{equation} 
One checks that such a sequence is split, hence $G \simeq G(32, \, 50) \times \mathbb{Z}_2$ and so $G$ is non-monolithic.  Thus, by Proposition \ref{prop:order-64},  every structure on $G$ is obtained by lifting a structure from its extra-special quotient $G(32, \, 50)$. 

As in the previous case, the total number of liftings  on $G$ is $1152 \cdot 1920 \cdot 256$ and, for each diagonal double Kodaira structure on $G(32, \, 49)$, there are precisely $16$ liftings that do not generate $G$. So,  the total  number  of structures on $G$ is $1152 \cdot 1920 \cdot 240$. Moreover, all these structures are non-strong, since in all cases $|K_1|=|K_2|=32$.

Finally, since $|\operatorname{Aut}(G)|=61440$, the number of structures on $G$ modulo automorphisms is 
\begin{equation}
(1152 \cdot 1920 \cdot 240)/61440 = 8640. 
\end{equation} 
 
 \subsection{The case $G=(64, 249)$} \label{subsec:64_249}

The group $G=G(64, \, 249)$ is generated by six elements $x_1, \ldots, x_6$ subject to the following set of relations:
\begin{table}[H]
  \begin{center}
    \begin{tabularx}{0.8\textwidth}{@{}llllll@{}}
$x_1^2=x_5,$ &                       &                      &                     &                    &                    \\
$ x_2^2=1,$  & $[x_1, \, x_2]=1,$    &                      &                     &                    &                    \\
$x_3^2=1,$   & $[x_1, \, x_3]=1,$    & $[x_2, \, x_3]=x_6,$ &                     &                    &                    \\
$ x_4^2=1,$  & $[x_1, \, x_4]=x_6, $ & $[x_2, \, x_4]=1, $  & $[x_3, \, x_4]=1, $ &                    &                    \\
$x_5^2=x_6,$ & $[x_1, \, x_5]=1,$    & $[x_2, \, x_5]=1,$   & $[x_3, \, x_5]=1,$  & $[x_4, \, x_5]=1,$ &                    \\
$x_6^2=1,$   & $[x_1, \, x_6]=1,$    & $[x_2, \, x_6]=1,$   & $[x_3, \, x_6]=1,$  & $[x_4, \, x_6]=1,$ & $[x_5, \, x_6]=1.$
\end{tabularx}
  \end{center}
\end{table}

We have $Z(G) = \langle x_5 \rangle \simeq \mathbb{Z}_4$ and $[G, \, G]= \langle x_6 \rangle \simeq \mathbb{Z}_2$. Moreover, $G$ is monolithic, with $\operatorname{mon}(G)= [G, \, G]$.  

No quotient of $G$ is isomorphic to either $G(32, \, 49)$ or $G(32, \, 50)$, so no structure on $G$ is obtained as a  lifting  from a quotient. In this sense, the structures on $G$ are really ``new''. The  computation with \verb|GAP4| shows that there are precisely 368640  structures on $G$, modulo the action of $\operatorname{Aut}(G)$. All these structures are strong.

\subsection{The case $G=(64, 266)$}

The group $G=G(64, \, 266)$ is generated by six elements $x_1, \ldots, x_6$ subject to the following set of relations:
\begin{table}[H]
  \begin{center}
    \begin{tabularx}{0.8\textwidth}{@{}llllll@{}}
$x_1^2=1,$   &                       &                      &                     &                    &                    \\
$ x_2^2=1,$  & $[x_1, \, x_2]=x_6,$  &                      &                     &                    &                    \\
$x_3^2=1,$   & $[x_1, \, x_3]=1,$    & $[x_2, \, x_3]=x_6,$ &                     &                    &                    \\
$ x_4^2=1,$  & $[x_1, \, x_4]=x_6, $ & $[x_2, \, x_4]=1, $  & $[x_3, \, x_4]=1, $ &                    &                    \\
$x_5^2=x_6,$ & $[x_1, \, x_5]=1,$    & $[x_2, \, x_5]=1,$   & $[x_3, \, x_5]=1,$  & $[x_4, \, x_5]=1,$ &                    \\
$x_6^2=1,$   & $[x_1, \, x_6]=1,$    & $[x_2, \, x_6]=1,$   & $[x_3, \, x_6]=1,$  & $[x_4, \, x_6]=1,$ & $[x_5, \, x_6]=1.$
\end{tabularx}
  \end{center}
\end{table}

We have $Z(G) = \langle x_5 \rangle \simeq \mathbb{Z}_4$ and $[G, \, G]= \langle x_6 \rangle \simeq \mathbb{Z}_2$. Moreover, $G$ is monolithic, with $\operatorname{mon}(G)= [G, \, G]$.  

No quotient of $G$ is isomorphic to either $G(32, \, 49)$ or $G(32, \, 50)$, so no structure on $G$ is obtained as a lifting from a quotient. In this sense, as in the previous case, the structures on $G$ are really ``new''. The computation with \verb|GAP4| shows that there are precisely 23040 structures on $G$, modulo the action of $\operatorname{Aut}(G)$. All these structures are non-strong, with $|K_1|=|K_2|=32$.

\section{The computation of the first homology group of $S$} \label{sec:first_homology}
We end this section by computing the first homology group $H_1(S, \,
\mathbb{Z})$, where $\mathbf{f} \colon S \to \Sigma_2 \times \Sigma_2$ is the
$G$-cover associated with a diagonal double Kodaira
structure of type $(b, \, n)=(2, \, 2)$ on an extra-special group of order
$32$. To this purpose, we will make use of the following result, which follows from \cite[Theorem p. 254]{Fox57}.

\begin{proposition} \label{prop:fundamental-group-branched-cover}
  Let $G$ be a finite group, and $\varphi \colon \mathsf{P}_2(\Sigma_b)
  \to G$ be an admissible epimorphism such that $\varphi(A_{12})$
  has order $n \geq 2$. If $\mathbf{f} \colon S \to \Sigma_b \times \Sigma_b$
  is the $G$-cover associated with $\varphi$, then
  $\pi_1(S)$ fits into a short exact sequence
  \begin{equation} \label{eq:fundamental-group-branched-cover}
    1 \to \pi_1(S) \to	\mathsf{P}_2(\Sigma_b)^{\operatorname{orb}}
    \stackrel{\,\,\,\bar{\varphi}}{\to} G \to 1,
  \end{equation}
  where the \emph{orbifold braid group} $\mathsf{P}_2(\Sigma_b)^{\operatorname{orb}}$ is defined as the quotient of
  $\mathsf{P}_2(\Sigma_b)$ by the normal closure of the cyclic subgroup
  $\langle A_{12}^n \rangle$, and $\bar{\varphi} \colon
  \mathsf{P}_2(\Sigma_b)^{\operatorname{orb}}  \to G$ is the group epimorphism
  naturally induced by $\varphi$.
\end{proposition}
Proposition \ref{prop:fundamental-group-branched-cover} allows one to compute $\pi_1(S)$, and so its abelianization $H_1(S,
\, \mathbb{Z})$. However, doing all the calculations by hand seems quite
difficult, so we resorted again to \verb|GAP4|. Part 7 of the \verb|GAP4| script in  the repository \cite{GAP4DDKS} contains the script used to analyze the group $G(64, \, 199)$; the scripts analyzing the other groups are similar.
Let us summarize our  results.

\begin{theorem} \label{thm:H1 in order 64}
  Let $\mathbf{f} \colon S \to \Sigma_2 \times \Sigma_2$ be the $G$-cover  associated with a diagonal double Kodaira structure of type $(b,\, n)=(2, \, 2)$ on a finite group $G$ of order $64$, and let $f \colon S \to \Sigma_{b_1} \times \Sigma_{b_2}$ be  the corresponding  diagonal double Kodaira fibration. Then the occurrences for $H_1(S, \, \mathbb{Z})$, $q(S)= \frac{1}{2} \operatorname{rank} \, H_1(S, \, \mathbb{Z})$, $K_S^2$, $c_2(S)$, $\sigma(S)$, $b_1$, $b_2$, $g_1$, $g_2$ are as in the table below.
  \begin{table}[H]
  \begin{center}
    \begin{tabularx}{0.85\textwidth}{@{}cccccccccc@{}}
\toprule
 $\mathrm{IdSmallGroup}(G)$ & $H_1(S, \, \mathbb{Z})$ & $q(S)$ & $K_S^2$ & $c_2(S)$ & $\sigma(S)$ & $b_1$ & $b_2$ & $g_1$ & $g_2$ \\
\toprule
 $G(64, \, 199)$ & $\mathbb{Z}^8 \oplus (\mathbb{Z}_2)^4$ & $4$ & $736$ & $320$ & $32$ & $2$ & $2$ & $81$ & $81$ \\
$G(64, \, 200)$ & $\mathbb{Z}^8 \oplus (\mathbb{Z}_2)^4$ & $4$ & $736$ & $320$ & $32$ & $2$ & $2$ & $81$ & $81$ \\
 $G(64, \, 201)$ & $\mathbb{Z}^8 \oplus (\mathbb{Z}_2)^4$ & $4$ & $736$ & $320$ & $32$ & $2$ & $2$ & $81$ & $81$ \\
 $G(64, \, 264)$ & $\mathbb{Z}^{12} \oplus (\mathbb{Z}_2)^3$ & $6$ & $736$ & $320$ & $32$ & $3$ & $3$ & $41$ & $41$\\   
 $G(64, \, 265)$ & $\mathbb{Z}^{12} \oplus (\mathbb{Z}_2)^3$ & $6$& $736$ & $320$ & $32$ & $3$ & $3$ & $41$ & $41$\\   
 $G(64, \, 249)$ & $\mathbb{Z}^8 \oplus (\mathbb{Z}_2)^4$ & $4$  & $736$ & $320$ & $32$ & $2$ & $2$ & $81$ & $81$ \\
 $G(64, \, 266)$ & $\mathbb{Z}^{12} \oplus (\mathbb{Z}_2)^3$ & $6$ & $736$ & $320$ & $32$ & $3$ & $3$ & $41$ & $41$\\   
$G(64, \, 266)$  & $\mathbb{Z}^{12} \oplus (\mathbb{Z}_2)^2 \oplus \mathbb{Z}_4$ & $6$ & $736$ & $320$ & $32$ & $3$ & $3$ & $41$ & $41$\\   
\bottomrule

\end{tabularx}
  \end{center}
\end{table}
More precisely, in the case $G=G(64, \, 266)$, up to the natural action of $\operatorname{Aut}(G)$ there exist $17280$ diagonal double Kodaira structures such that $H_1(S, \, \mathbb{Z})=\mathbb{Z}^{12} \oplus (\mathbb{Z}_2)^3$ and $5760$ diagonal double Kodaira structures such that $H_1(S, \, \mathbb{Z})=\mathbb{Z}^{12} \oplus (\mathbb{Z}_2)^2 \oplus \mathbb{Z}_4$. 
\end{theorem}

Thus, in the case  $G=(64,\, 266)$ there are two occurrences for the torsion part of $H_1(S, \, \mathbb{Z})$. Moreover, \emph{every} curve of genus $b=2$ can be used as a starting point for our construction, and so such a construction depends on $3$ parameters. We can therefore state the following

\begin{corollary} \label{cor:non-homotopically equivalent}
There exist  two $3$-dimensional families $\mathcal{F}_1$, $\mathcal{F}_2$ of Kodaira doubly fibred surfaces with $b_1=b_2=3$, $g_1=g_2=41$, $\sigma(S)=32$ such that  
\begin{itemize}
\item the surfaces in $\mathcal{F}_1$ and those in $\mathcal{F}_2$ have the same biregular invariants
\begin{equation}
p_g(S)=93, \; \; q(S)=6, \; \; K_S^2=736
\end{equation}
\item the surfaces in $\mathcal{F}_1$ and those in $\mathcal{F}_2$ have the same Betti numbers
\begin{equation}
\mathsf{b}_0=\mathsf{b}_4=1, \; \; \mathsf{b}_1=\mathsf{b}_3=12, \; \; \mathsf{b}_2=342
\end{equation}
\item the surfaces in $\mathcal{F}_1$ and those in $\mathcal{F}_2$ have different torsion part of $H_1(S, \, \mathbb{Z})$. In particular, they are not homotopically equivalent.
\end{itemize}

\begin{remark} \label{rem:first occurrence}
The ones described in Corollary \ref{cor:non-homotopically equivalent} are, to our knowledge, the first explicit examples of positive-dimensional families of doubly fibred Kodaira surfaces having the same base and fibre genera, the same biregular invariants, the same Betti numbers and different fundamental group.

We do not know if the same phenomenon occurs when $|G|=32$. The computation with \verb|GAP4| shows that all the surfaces in the two families constructed in \cite[Section 5]{PolSab22}, and corresponding to the extra-special groups $\mathsf{H}_5(\mathbb{Z}_2)$ and $\mathsf{G}_5(\mathbb{Z}_2)$, satisfy $H_1(S, \, \mathbb{Z})=\mathbb{Z}^8 \oplus (\mathbb{Z}_2)^4$.  Therefore  the torsion part in $H_1(S, \, \mathbb{Z})$ does not distinguish them; it seems an interesting  problem to establish whether the fundamental groups of surfaces in different families are isomorphic or not. 
\end{remark}

\end{corollary}

\section*{Acknowledgements}
Francesco Polizzi was partially supported by GNSAGA-INdAM.


\newpage

\section*{Appendix. Non-abelian groups \texorpdfstring{$G$}{G}
of order   \texorpdfstring{$|G| = 36, \, 40, \, 48, \, 54, \, 56, \, 60$}{|G|
= 36, \, 40, \, 48, \, 54, \, 56, \, 60}}
  \label{Appendix_A}

\vskip 0.5cm

Source: \url{https://people.maths.bris.ac.uk/~matyd/GroupNames/}

\vskip 0.5cm

\begin{table}[H]
\caption[caption]{Non-abelian groups of
    order $36$.\\
 }
  \label{table:36-nonabelian}
    \centering
    \begin{tabularx}{0.85\linewidth}{@{}ccc@{}}
      \toprule
      $\mathrm{IdSmallGroup}(G)$ & $G$ &
      $\mathrm{Presentation}$ \\
      \toprule
      $G(36, \, 1)$ & $\mathsf{D}_{4, \, 9, \, -1}$ &
      $\langle x, \, y \; | \;
      x^4=y^9=1, \, xyx^{-1}=y^{-1} \rangle$	\\
&& \\
      $G(36, \, 3)$ & $(\mathbb{Z}_2)^2 \rtimes \mathbb{Z}_9$ &
      $\langle a, \, b,\, x \; | \;
      a^2=b^2=x^9=1, \, [a, \, b]=1, $ \\
      & & $xax^{-1}=b, \,  xbx^{-1}=ab \rangle$ \\
&& \\
	$G(36, \, 4)$ & $\mathsf{D}_{36}$ &
      $\langle x, \, y \; | \;
      x^2=y^{18}=1, \, xyx^{-1}=y^{-1} \rangle$ \\
      && \\
$G(36, \, 6)$ & $\mathsf{D}_{4, \, 3, \, -1} \times \mathbb{Z}_3$ &
     $\langle x, \, y \; | \; x^4=y^3=1, \, xyx^{-1} = y^{-1} \rangle \times
     \langle z \; | \; z^3=1 \rangle$ \\
&& \\
      $G(36, \, 7)$ & $(\mathbb{Z}_3)^2  \rtimes \mathbb{Z}_4$ &
      $\langle a, \, b,\, x \; | \;
      a^3=b^3=x^4=1, \, [a, \, b]=1, $ \\
	& & $xax^{-1}=a^{-1}, \,  xbx^{-1}=b^{-1} \rangle$	\\
&& \\
$G(36, \, 9)$ & $(\mathbb{Z}_3)^2 \rtimes \mathbb{Z}_4$ &
      $\langle a, \, b,\, x \; | \;
      a^3=b^3=x^4=1, \, [a, \, b]=1, $ \\
	& & $xax^{-1}=b, \,  xbx^{-1}=a^{-1} \rangle$	\\
&& \\
$G(36, \, 10)$ & $\mathsf{S}_3 \times \mathsf{S}_3$ & $ \langle (1, \,2),
(1,\, 2, \, 3), \, (4, \,5), (4,\, 5, \, 6)\rangle $\\
&& \\
$G(36, \, 11)$ & $\mathsf{A}_4	\times \mathbb{Z}_3$ &$ \langle (1 \, 2)(3 \,
4), \, (1\, 2\, 3) \rangle \times   \langle z \; | \; z^3=1  \rangle $\\
&& \\
$G(36, \, 12)$ & $\mathsf{D}_{12}  \times \mathbb{Z}_3$ & $ \langle x, \,
y  \; | \; x^2=y^6=1, \, xyx^{-1} = y^{-1} \rangle \times   \langle z \; |
\; z^3=1  \rangle $\\
&& \\
$G(36, \, 13)$ &  $((\mathbb{Z}_3)^2  \rtimes \mathbb{Z}_2) \times
\mathbb{Z}_2$ &
$\langle a, \, b, \, c, \, x \; | \; a^3=b^3=c^2=x^2=1,$ \\
& & $[a, \, b]=1, \, [b, \, c]=1, \, [x, \, c]=1,$ \\
& &   $xax^{-1}=a^{-1}, \,  xbx^{-1}=b^{-1} \rangle$ \\
\bottomrule
\end{tabularx}
  \end{table}

\newpage

\begin{table}[H]
    \caption[caption]{Non-abelian groups of
    order $40$.\\
  }
  \label{table:40-nonabelian}
    \centering
    \begin{tabularx}{0.85\linewidth}{@{}ccc@{}}
      \toprule
      $\mathrm{IdSmallGroup}(G)$ & $G$ &
      $\mathrm{Presentation}$ \\
      \toprule
     $G(40,\,1)$ & $\mathsf{D}_{8, \, 5, \, -1}$ & $\langle x, \, y \; | \;
     x^8=y^5=1, \, xyx^{-1} = y^{-1} \rangle$ \\
     && \\
      $G(40,\,3)$ & $\mathsf{D}_{8, \, 5, \, 2}$ & $\langle x, \, y \; | \;
      x^8=y^5=1, \, xyx^{-1} = y^2 \rangle$ \\
     && \\
$G(40, \, 4)$ & $\mathbb{Z}_5 \rtimes \mathsf{Q}_8$ & $\langle a, \, i,
\, j  \; | \; a^5=i^4=j^4=1,$ \\ && $ iji=j, \,  jij=i, \, [i, \, a]=1, \,
jaj^{-1}=a^{-1} \rangle$ \\
&& \\
$G(40, \, 5)$ & $\mathsf{D}_{10} \times \mathbb{Z}_4$ & $\langle x, \, y \;
| \; x^2=y^5=1, \, xyx^{-1}=y^{-1}  \rangle \times \langle z \; | \; z^4=1
\rangle$ \\
&& \\
$G(40, \, 6)$ & $\mathsf{D}_{40} $ & $\langle x, \, y \; | \; x^2=y^{20}=1,
\, xyx^{-1}=y^{-1} \rangle $  \\
&& \\
$G(40, \, 7)$ &  $\mathsf{D}_{4, \, 5, \, -1}	\times \mathbb{Z}_2$ &
$\langle x, \, y \; | \; x^4=y^5=1, \, xyx^{-1}=y^{-1}	\rangle \times
\langle z \; | \; z^2=1 \rangle$ \\
&& \\
$G(40, \, 8)$ & $(\mathbb{Z}_2 \times \mathbb{Z}_{10}) \rtimes \mathbb{Z}_2$
&  $\langle x, \, y, \, z \; | \; x^2=y^2=z^5=1, $ \\ && $(xz)^2=1, \,
(yx)^4=1, \, [y, \, z]=1  \rangle$ \\
&& \\
$G(40, \, 10)$ & $\mathsf{D}_8 \times \mathbb{Z}_5$ &  $\langle x, \, y \;
| \; x^2=y^4=1, \, xyx^{-1}=y^{-1}  \rangle \times \langle z \; | \; z^5=1
\rangle$ \\
&& \\
$G(40, \, 11)$ & $\mathsf{Q}_8 \times \mathbb{Z}_5$ & $ \langle i,\,j,\,k
\mid i^2 = j^2 =
      k^2 = ijk \rangle  \times \langle z \; | \; z^5=1 \rangle$ \\
&& \\
$G(40, \, 12)$ & $\mathsf{D}_{4, \, 5, \, 2}   \times \mathbb{Z}_2$ & $\langle
x, \, y \; | \; x^4=y^5=1, \, xyx^{-1}=y^2  \rangle \times \langle z \; |
\; z^2=1 \rangle$ \\
&& \\
$G(40, \, 13)$ & $\mathsf{D}_{10}   \times \mathbb{Z}_2 \times \mathbb{Z}_2$
& $\langle x, \, y \; | \; x^2=y^5=1, \, xyx^{-1}=y^{-1}  \rangle $ \\ &&
$ \times  \,  \langle z, \, w \; | \; z^2=w^2=[z, \, w]=1 \rangle$ \\
&& \\
\bottomrule
\end{tabularx}
  \end{table}

\newpage

\begin{xltabular}{0.55\linewidth}{@{}ccc@{}}
	\caption{Non-abelian groups of
	    order $48$. } \\
	    \label{table:48-nonabelian} \\
	\toprule
	    $\mathrm{IdSmallGroup}(G)$ & $G$ &
      $\mathrm{Presentation}$ \\
      \toprule
\endfirsthead

	\toprule
	    $\mathrm{IdSmallGroup}(G)$ & $G$ &
      $\mathrm{Presentation}$ \\
      \toprule
\endhead

\midrule
\multicolumn{2}{r@{}}{\itshape \ldots continues on next page}\\
\endfoot
\midrule
\endlastfoot

 $G(48,\,1)$ & $\mathsf{D}_{16, \, 3, \, -1}$ & $\langle x, \, y \; |
     \; x^{16}=y^3=1, \, xyx^{-1} = y^{-1} \rangle$ \\
     && \\
      $G(48,\,3)$ & $(\mathbb{Z}_4)^2  \rtimes \mathbb{Z}_3$ & $\langle a,
      \, b, \, x \; | \; a^4=b^4=x^3=1,$ \\
      && $[a, \, b]=1, \, xax^{-1}=b, \, xbx^{-1}=a^3b^3 \rangle $ \\
&& \\
$G(48, \, 4)$ &   $\mathsf{S}_3 \times \mathbb{Z}_8$ & $ \langle (1\, 2), \,
(1\, 2\, 3) \rangle \times \langle z \; | \; z^8=1 \rangle$ \\
  && \\
$G(48,\,5)$ & $\mathsf{D}_{2, \, 24, \, 5}$ & $\langle x, \, y \; | \;
x^2=y^{24}=1, \, xyx^{-1} = y^5  \rangle$ \\
     && \\
  $G(48,\,6)$ & $\mathsf{D}_{2, \, 24, \, 11}$ & $\langle x, \, y \; | \;
  x^2=y^{24}=1, \, xyx^{-1} = y^{11}  \rangle$ \\
     && \\
$G(48,\,7)$ & $\mathsf{D}_{48}$ & $\langle x, \, y \; | \; x^2=y^{24}=1, \,
xyx^{-1} = y^{-1}  \rangle$ \\
&& \\
$G(48, \, 8)$ & $\mathbb{Z}_3 \rtimes \mathsf{Q}_{16}$ & $\langle a, \, x,
\, y \; | \; a^3=x^8=y^4=1, \, x^4=y^2, $ \\
&& $[x, \, a]=1, \, yay^{-1}=a^{-1}, \, yxy^{-1}=x^{-1} \rangle$ \\
     && \\
 $G(48, \, 9)$ & $\mathsf{D}_{8, \, 3, \, -1} \times \mathbb{Z}_2$ & $\langle
 x, \, y \; | \; x^8=y^3=1, \, xyx^{-1} = y^{-1} \rangle \times     \langle
 z \; | \; z^2=1 \rangle$ \\
 && \\
$G(48, \, 10)$ & $\mathsf{D}_{8, \, 3, \, -1}  \rtimes \mathbb{Z}_2$ &
$\langle x, \, y, \, z \; | \; x^8=y^3=z^2=1,$ \\
&& $ xyx^{-1}=y^{-1}, \, zxz^{-1}=x^5, \,  [z, \, y]=1 \rangle$ \\
&& \\
$G(48, \, 11)$ & $\mathsf{D}_{4, \, 3, \, -1}  \times \mathbb{Z}_4$ &
$\langle x, \, y \; | \; x^4=y^3=1, \, xyx^{-1} = y^{-1} \rangle \times
\langle z \; | \; z^4=1 \rangle$ \\
 && \\
$G(48, \, 12)$ & $\mathsf{D}_{4, \, 3, \, -1}  \rtimes \mathbb{Z}_4$ &
$\langle x, \, y, \, z \; | \; x^4=y^3=z^4=1, $ \\
&& $ xyx^{-1} = y^{-1}, \, zxz^{-1}=x^{-1}, \, [z, \, y]=1 \rangle$ \\
&& \\
$G(48,\,13)$ & $\mathsf{D}_{4, \, 12, \, -1}$ & $\langle x, \, y \; | \;
x^4=y^{12}=1, \, xyx^{-1} = y^{-1} \rangle$ \\
&& \\
$G(48,\,14)$ & $(\mathbb{Z}_2  \times \mathbb{Z}_{12}) \rtimes \mathbb{Z}_2$
& $\langle a, \, b, \, x \; | \; a^2=b^{12}=x^2=1,$ \\
      && $[a, \, b]=1, \, [x, \, a]=1, \, xbx^{-1}=ab^5 \rangle $ \\
&& \\
$G(48, \, 15)$ & $\mathsf{D}_{2, \, 12, \, 7}  \rtimes \mathbb{Z}_2$ &
$\langle x, \, y, \, z \; | \; x^2=y^{12}=z^2=1,$ \\
&& $ xyx^{-1}=y^7, \, zxz^{-1}=xy^9, \, zyz^{-1}=y^{-1} \rangle$ \\
&& \\
$G(48, \, 16)$ & $(\mathbb{Z}_3 \rtimes \mathsf{Q}_8) \rtimes \mathbb{Z}_2$
& $\langle  x, \, y, \, z \; | \; x^4=y^2=z^3=1, $\\
&& $x^{-1}zx=z^{-1}, \,z^{-1}yz=y^{-1},$ \\
&& $(x^{-1}yx^{-1})^2=1, \, (yx)^3yx^{-1} = 1 \rangle$ \\
     && \\
$G(48, \, 17)$	& $(\mathsf{Q}_8 \times \mathbb{Z}_3)  \rtimes \mathbb{Z}_2$ &
$\langle i, \, j, \, a, \, z \; | \; i^4=j^4=a^3=z^2=1, $\\
&& $(ij)^4=1, \, iji=j, \, jij=1, $\\
&& $[a, \, i]=1, \, [a, \, j]=1, $\\
&& $ziz^{-1}=i^{-1}, \, zjz^{-1}=ji, \, zaz^{-1}=a^{-1} \rangle$ \\
 && \\
 $G(48, \, 18)$ & $\mathbb{Z}_3 \rtimes \mathsf{Q}_{16}$ & $\langle a, \,
 x, \, y \; | \; a^3=x^8=y^4=1, \, x^4=y^2, $ \\
&& $xax^{-1}=a^{-1}, \, yay^{-1}=a^{-1}, \, yxy^{-1}=x^{-1} \rangle$ \\
     && \\
      $G(48, \, 19)$ & $(\mathbb{Z}_2 \times \mathbb{Z}_6)   \rtimes
      \mathbb{Z}_4$ & $\langle a, \, b, \, x \; | \;  a^2=b^6=x^4=1, $ \\
  && $[a, \, b]=1, \, xax^{-1} = a,  \, xbx^{-1} = ab^{-1} \rangle$\\
  && \\
 $G(48, \, 21)$ & $(\mathbb{Z}_2 \times \mathbb{Z}_6)	\rtimes \mathbb{Z}_4$
 & $\langle a, \, b, \, x \; | \;  a^2=b^6=x^4=1, $ \\
  && $[a, \, b]=1, \, xax^{-1} = ab^3,	\, [x, \, b]=1 \rangle$\\
$G(48, \, 22)$ & $\mathsf{D}_{4, \, 4, \, -1}  \times \mathbb{Z}_3$ &
$\langle x, \, y \; | \; x^4=y^4=1, \, xyx^{-1} = y^{-1} \rangle \times
\langle z \; | \; z^3=1 \rangle$ \\
 && \\
 $G(48, \, 24)$ & $\mathsf{D}_{2, \, 8, \, 5}  \times \mathbb{Z}_3$ &
 $\langle x, \, y \; | \; x^2=y^8=1, \, xyx^{-1} = y^5	\rangle  \times
 \langle z \; | \; z^3=1 \rangle$ \\
 && \\
$G(48, \, 25)$ & $\mathsf{D}_{16}  \times \mathbb{Z}_3$ & $\langle x, \,
y \; | \; x^2=y^8=1, \, xyx^{-1} = y^{-1}  \rangle  \times     \langle z \;
| \; z^3=1 \rangle$ \\
 && \\
$G(48, \, 26)$ & $\mathsf{QD}_{16}  \times \mathbb{Z}_3$ & $\langle x, \,
y \; | \; x^2=y^8=1, \, xyx^{-1} = y^3	\rangle  \times     \langle z \; |
\; z^3=1 \rangle$ \\
 && \\
  $G(48, \, 27)$ & $\mathsf{Q}_{16}  \times \mathbb{Z}_3$ & $\langle x,
  \, y, \, z\; | \; x^4=y^2=z^2=xyz  \rangle  \times	 \langle w \; | \;
  w^3=1 \rangle$ \\
 && \\
$G(48, \, 28)$ & $\mathsf{SL}(2, \, \mathbb{F}_3) \, . \, \mathbb{Z}_2$ &
$\langle x, \, y, \, z\; | \; x^4=y^3=z^2=xyz  \rangle$ \\
 && \\
 $G(48, \, 29)$ & $\mathsf{GL}(2, \, \mathbb{F}_3)$ & $\langle x, \, y, \,
 z\; | \; x^2=y^3=z^4=1, $ \\
 && $(xy)^{2}=1, \, (z^{-1}y)^3=1, \, (y^{-1}xz^{-1})^2=1,$ \\
 && $[z^2, \, x]=1, \, [z^2, \, y]=1 \rangle$ \\
&& \\
$G(48, \, 30)$ & $\mathsf{A}_4 \rtimes \mathbb{Z}_4$ & $\langle x, \, y, \,
z \; | \; x^4=y^3=z^2=1, $\\
&& $x^{-1}yx=y^{-1}, (zx^{-2})^2=1, (y^{-1}z)^3=1,$ \\
&& $y^{-1}zyxzx^{-1}=1 \rangle$ \\
&& \\
$G(48, \, 31)$ & $\mathsf{A}_4	\times \mathbb{Z}_4$ &$ \langle (1 \, 2)(3 \,
4), \, (1\, 2\, 3) \rangle \times   \langle z \; | \; z^4=1  \rangle $\\
&& \\
$G(48, \, 32)$ & $\mathsf{SL}(2, \, \mathbb{F}_3) \times \mathbb{Z}_2$ &
$\langle x, \, y, \, z\; | \; x^3=y^3=z^2=xyz  \rangle \times \langle w \;
| \; w^2=1 \rangle$ \\ \\
 && \\
$G(48, \, 33)$ & $((\mathbb{Z}_2 \times \mathbb{Z}_4) \rtimes \mathbb{Z}_2)
\rtimes \mathbb{Z}_3$ &
$\langle x, \, y \; | \; x^3=y^3=z^2=w^2=xyz, $ \\
&& $[x, \, w]=1, \, [y, \, w]=1, [z, \, w]=1 \rangle$ \\
&& \\
$G(48, \, 34)$ & $\mathsf{Q}_{24}  \times \mathbb{Z}_2$ & $\langle x, \,
y, \, z\; | \; x^6=y^2=z^2=xyz	\rangle  \times     \langle w \; | \; w^2=1
\rangle$ \\
&& \\
 $G(48, \, 35)$ & $\mathsf{S}_3  \times \mathbb{Z}_2 \times \mathbb{Z}_4$
 & $\langle (1 \, 2), \, (1 \, 2\, 3) \rangle \times	 \langle z \; | \;
 z^2=1 \rangle \times \langle w \; | \; w^4=1 \rangle $ \\
& \\
$G(48, \, 36)$ & $\mathsf{D}_{24}  \times \mathbb{Z}_2$ & $\langle x, \, y \;
| \; x^2=y^{12}=1, \, xyx^{-1} = y^{-1}  \rangle  \times     \langle z \;
| \; z^2=1 \rangle$ \\
 && \\
 $G(48, \, 37)$ & $(\mathbb{Z}_2 \times \mathbb{Z}_{12})   \rtimes
 \mathbb{Z}_2$ & $\langle a, \, b, \, x \; | \;  a^2=b^{12}=x^2=1, $ \\
    && $[a, \, b]=1, \, xax^{-1} = ab^6,  \, xbx^{-1} = b^5 \rangle$\\
&& \\
 $G(48, \, 38)$ & $\mathsf{S}_3 \times \mathsf{D}_8$ &	$\langle (1 \, 2), \,
 (1 \, 2\, 3) \rangle \times \langle x, \, y \; | \; x^2=y^4=1, \, xyx^{-1}
 = y^{-1}  \rangle$ \\
&  &\\
$G(48, \, 39)$ & $(\mathsf{S}_3 \times \mathbb{Z}_4) \rtimes \mathbb{Z}_2$
&  $\langle a, \, b, \, c, \, d \; | \; a^2=b^4=c^2=d^3=1,$ \\
&& $[a, \, b]=1, \, [b, \, d]=1, \,  [c, \, d]=1, $ \\
&& $(ad)^2=1, \, (cb^{-1})^2=1, \, b^2(ca)^2=1 \rangle$ \\
&& \\
$G(48, \, 40)$ & $\mathsf{S}_3 \times \mathsf{Q}_8$ &  $\langle (1 \, 2), \,
(1 \, 2\, 3) \rangle \times \langle i, \, j, \, k\; | \; i^2=j^2=k^2=ijk
\rangle$  \\
&  &\\
$G(48, \, 41)$ & $(\mathsf{S}_3 \times \mathbb{Z}_4) \rtimes \mathbb{Z}_2$
&  $\langle a, \, b, \, c \; | \; a^{12}=b^4=c^2=1, $ \\
&& $a^6=b^2, \, [b, \, c]=1, $ \\
&& $bab^{-1}=a^7, \, cac^{-1}=a^{-1} \rangle$ \\
& & \\
$G(48, \, 42)$ & $(\mathbb{Z}_2 \times \mathbb{Z}_6)   \rtimes \mathbb{Z}_4$
& $\langle a, \, b, \, x \; | \;  a^2=b^6=x^4=1, $ \\
    && $[a, \, b]=1, \, xax^{-1} = a,  \, xbx^{-1} = b^{-1} \rangle$\\
$G(48, \, 43)$ & $(((\mathbb{Z}_2)^2 \times \mathbb{Z}_3) \rtimes \mathbb{Z}_2)
\times \mathbb{Z}_2$ & $\langle x, \, y, \, z, \, w, \, t \; | \; x^2 =y^2=
z^2=  w^3 =t^2 =1, $\\
&& $[y, \, z]=1, \, [y,\, w]=1, \, [z, \, w]=1, $ \\
&&  $[x, \, t]=1, \, [y, \, t]=1, \ [z, \, t]=1, \, [w, \, t]=1, $ \\
&& $xyx^{-1}=y, \, xzx^{-1}=zy, \, xwx^{-1}=w^{-1} \rangle$ \\
&& \\
$G(48, \, 45)$ & $\mathsf{D}_8	\times \mathbb{Z}_6$ & $\langle x, \, y \;
| \; x^2=y^4=1, \, xyx^{-1} = y^{-1}  \rangle  \times	  \langle z \; | \;
z^6=1 \rangle$ \\
 && \\
 $G(48, \, 46)$ & $\mathsf{Q}_8  \times \mathbb{Z}_6$ & $\langle i, \, j,
 \, k\; | \; i^2=j^2=k^2=ijk  \rangle  \times	  \langle z \; | \; z^6=1
 \rangle$ \\
 && \\
 $G(48, \, 47)$ & $((\mathbb{Z}_2 \times \mathbb{Z}_4) \rtimes
 \mathbb{Z}_2) \times \mathbb{Z}_3 $ & $\langle x, \, y, \, z, \, w \; | \;
 x^2=y^2=z^4=w^3=1, $ \\
 && $ [x, \, w]=1, \, [y, \, w]=1, \, [z, \, w]=1, $ \\
 && $[x, \, z]=1, \, [y, \, z]=1, \, z^2(yx)^2=1 \rangle$ \\
&& \\
$G(48, \, 48)$ & $\mathsf{S}_4 \times \mathbb{Z}_2$ &  $\langle (1 \, 2), \,
(1 \, 2 \, 3 \, 4) \rangle \times   \langle z \; | \; z^2=1 \rangle$	\\
 && \\
$G(48, \, 49)$ & $\mathsf{A}_4 \times \mathbb{Z}_2$ &  $\langle (1 \, 2)(3 \,
4), \, (1 \, 2 \, 3) \rangle \times  \langle x, \, y \; | \; x^2=y^2=1, \,
[x,\, y]=1 \rangle$    \\
 && \\
$G(48, \,50)$ & $(\mathbb{Z}_2)^4 \rtimes \mathbb{Z}_3$ &  $\langle a,\, b,
\, c, \, d, \, x \; | \; a^2=b^2=c^2=d^2= x^3=1, $ \\
&& $[a, \, b]=1, \, [a, \, c]=1, \, [a, \, d]=1, $ \\
&& $[b, \, c]=1, \, [b, \, d]=1, \, [c, \, d]=1, $ \\
&& $xax^{-1}= d, \, xbx^{-1}=bc, \, xcx^{-1}=b, \, xdx^{-1}=ad \rangle$ \\
&& \\
$G(48, \,51)$ & $\mathsf{S}_3 \times (\mathbb{Z}_2)^3$ &  $\langle (1\, 2), \,
(1 \, 2 \, 3) \rangle $ \\
&&  $\times \, \langle x, \, y, \, z \; | \; x^2=y^2=z^2=1,$ \\
&&  $[x, \, y]=1, \, [x, \, z]=1, \, [y, \, z]=1 \rangle $ \\
\bottomrule
\end{xltabular}

\newpage

 \begin{table}[H]
     \caption[caption]{Non-abelian groups of
    order $54$. }
  \label{table:54-nonabelian}
  \begin{center}
    \begin{tabularx}{0.85\linewidth}{@{}ccc@{}}
      \toprule
$\mathrm{IdSmallGroup}(G)$ & $G$ &
      $\mathrm{Presentation}$ \\
      \toprule
$G(54, \, 1)$ & $\mathsf{D}_{54}$ & $\langle x, \, y \; | \; x^2 = y^{27}
=1, xyx^{-1}= y^{-1} \rangle$ \\
      && \\
$G(54, \, 3)$ & $\mathsf{D}_{18}  \times \mathbb{Z}_3$ & $ \langle x, \,
y  \; | \; x^2=y^9=1, \, xyx^{-1} = y^{-1} \rangle \times   \langle z \; |
\; z^3=1  \rangle $\\
  &&\\
$G(54, \, 4)$ &   $\mathsf{S}_3 \times \mathbb{Z}_9$ & $ \langle (1\, 2), \,
(1\, 2\, 3) \rangle \times \langle z \; | \; z^9=1 \rangle$ \\
  && \\
$G(54,\,5)$ & $(\mathbb{Z}_3)^2  \rtimes \mathbb{Z}_6$ & $\langle a, \, b,
\, x \; | \; a^3=b^3=x^6=1,$ \\
      && $[a, \, b]=1, \, xax^{-1}=a^{-1}b^{-1}, \, xbx^{-1}=b^{-1} \rangle
      $ \\
&& \\
 $G(54, \, 6)$ & $\mathsf{D}_{6, \, 9, \, 2}$ &
      $\langle x, \, y \; | \;
      x^6=y^9=1, \, xyx^{-1}=y^2 \rangle$	\\
&& \\
$G(54,\,7)$ & $(\mathbb{Z}_3  \times \mathbb{Z}_9) \rtimes \mathbb{Z}_2$ &
$\langle a, \, b, \, x \; | \; a^3=b^9=x^2=1,$ \\
      && $[a, \, b]=1, \, xax^{-1}=a^{-1}, \, xbx^{-1}=b^{-1} \rangle $ \\
&& \\
$G(54, \, 8)$ & $((\mathbb{Z}_3)^2 \rtimes \mathbb{Z}_3) \rtimes \mathbb{Z}_2$
& $\langle x, \, y, \, z, \, w \; | \; x^3=y^3=z^3=w^2=1, $ \\
 && $ [x, \, y]=1, \, zxz^{-1}=xy^{-1}, wxw^{-1}=x^{-1}, $ \\
 && $[y, \, z]=1, \, [y, \, w]=1, \, wzw^{-1}=z^{-1} \rangle$ \\
&& \\
$G(54, \, 10)$ & $((\mathbb{Z}_3)^2 \rtimes \mathbb{Z}_3) \times \mathbb{Z}_2$
& $\langle x, \, y, \, z, \, w \; | \; x^3=y^3=z^3=w^2=1, $ \\
 && $ [x, \, y]=1, \, zxz^{-1}=xy^{-1}, \, [y, \, z]=1,$ \\
 && $[w, \, x]=1, \,  [w, \, y]=1, \, [w, \, z]=1 \rangle$ \\
&& \\
$G(54, \, 11)$ & $\mathsf{D}_{3, \, 9, \, 4}  \times \mathbb{Z}_2$ & $\langle
x, \, y \; | \; x^3=y^9=1, \, xyx^{-1} = y^4  \rangle  \times	  \langle z \;
| \; z^2=1 \rangle$ \\
 && \\
&& \\
 $G(54, \, 12)$ & $\mathsf{S}_3 \times (\mathbb{Z}_3)^2$ &  $\langle (1 \,
 2), \, (1 \, 2\, 3) \rangle \times \langle x, \, y \; | \; x^3=y^3=1, \,
 [x, \, y]=1  \rangle$ \\
&&\\
$G(54, \, 13)$ & $(\mathbb{Z}_3)^2 \rtimes \mathsf{S}_3$ & $\langle a, \,
b, \, x, \, y \; | \; a^3=b^3=x^2=y^3=1, $\\
&& $[a, \, b]=1, \, xyx^{-1}=y^{-1}, $ \\
&& $[x, \, a]=1, \, xbx^{-1}=b^{-1}, $\\
&& $[y, \, a]=1, \, [y, \, b]=1 \rangle$ \\
&&\\
$G(54, \, 14)$ & $(\mathbb{Z}_3)^3 \rtimes \mathbb{Z}_2$ & $\langle a, \,
b, \, c, \, x \; | \; a^3=b^3=c^3=x^2=1, $\\
&& $[a, \, b]=1, \, [a, \, c]=1, \, [b, \, c]=1, $ \\
&& $xax^{-1}=a^{-1}, \,  xbx^{-1}=b^{-1}, \, xcx^{-1}=c^{-1} \rangle$\\

\bottomrule
	   \end{tabularx}
  \end{center}
\end{table}

\newpage

\begin{table}[H]
	    \caption[caption]{Non-abelian groups of
    order $56$.}
   \label{table:56-nonabelian}
  \begin{center}
    \begin{tabularx}{0.85\textwidth}{@{}ccc@{}}
      \toprule
$\mathrm{IdSmallGroup}(G)$ & $G$ &
      $\mathrm{Presentation}$ \\
      \toprule
 $G(56, \, 1)$ & $\mathsf{D}_{8, \, 7, \, -1}$ &
      $\langle x, \, y \; | \;
      x^8=y^7=1, \, xyx^{-1}=y^{-1} \rangle$	\\
&& \\
$G(56, \, 3)$ & $\mathbb{Z}_7 \rtimes \mathsf{Q}_8$ &
      $\langle x, \, y \; | \;
      y^{28}=1, \, x^2=y^{14}, \, xyx^{-1}=y^{-1} \rangle$	\\
&& \\
$G(56, \, 4)$ & $\mathsf{D}_{14}  \times \mathbb{Z}_4$ & $ \langle x, \,
y  \; | \; x^2=y^7=1, \, xyx^{-1} = y^{-1} \rangle \times   \langle z \; |
\; z^4=1  \rangle $\\
&& \\
$G(56, \, 5)$ & $\mathsf{D}_{56} $ & $ \langle x, \, y	\; | \; x^2=y^{28}=1,
\, xyx^{-1} = y^{-1} \rangle  $\\
&& \\
$G(56, \, 6)$ & $\mathsf{D}_{4, \, 14, \, -1}$ &
      $\langle x, \, y \; | \;
      x^4=y^{14}=1, \, xyx^{-1}=y^{-1} \rangle$ \\
&& \\
$G(56, \, 7)$ & $\mathbb{Z}_7 \rtimes \mathsf{D}_8$ &
      $\langle a, \, x, \, y \; | \; a^7=x^2=y^4=1, \, xyx^{-1}=y^{-1}, $\\
   &&	 $ xax^{-1}=a^{-1},  \, yay^{-1}=a^{-1}
      \rangle$	\\
&& \\
$G(56, \, 9)$ & $\mathsf{D}_{8}  \times \mathbb{Z}_7$ & $ \langle x, \,
y  \; | \; x^2=y^4=1, \, xyx^{-1} = y^{-1} \rangle \times   \langle z \; |
\; z^7=1  \rangle $\\
&& \\
$G(56, \, 10)$ & $\mathsf{Q}_8	\times \mathbb{Z}_7$ & $\langle i, \, j, \, k\;
| \; i^2=j^2=k^2=ijk  \rangle  \times	  \langle z \; | \; z^7=1 \rangle$ \\
 && \\
 $G(56, \, 11)$ & $(\mathbb{Z}_2)^3 \rtimes \mathbb{Z}_7$ & $\langle a, \,
 b, \, c, \, x \; | \; a^2=b^2=c^2=x^7=1, $\\
&& $[a, \, b]=1, \, [a, \, c]=1, \, [b, \, c]=1, $ \\
&& $xax^{-1}=bc, \,  xbx^{-1}=a, \, xcx^{-1}=b \rangle$\\
&& \\
$G(56, \,12)$ & $\mathsf{D}_{14} \times (\mathbb{Z}_2)^2$ &  $\langle x, \,
y  \; | \; x^2=y^7=1, \, xyx^{-1} = y^{-1} \rangle $ \\
&&  $\times \, \langle a, \, b \; | \; a^2=b^2=1, \, [a, \, b]=1 \rangle$ \\
\bottomrule
	   \end{tabularx}
  \end{center}
 \end{table}

\newpage

\begin{table}[H]
     \caption[caption]{Non-abelian groups of
    order $60$.}
   \label{table:60-nonabelian}
  \begin{center}
    \begin{tabularx}{0.85\textwidth}{@{}ccc@{}}
      \toprule
      $\mathrm{IdSmallGroup}(G)$ & $G$ &
      $\mathrm{Presentation}$ \\
      \toprule
   $G(60, \, 1)$ & $\mathsf{D}_{4, \, 3, \, -1}   \times \mathbb{Z}_5$ &
   $\langle x, \, y \; | \; x^4=y^3=1, \, xyx^{-1}=y^{-1}  \rangle \times
   \langle z \; | \; z^5=1 \rangle$ \\
&& \\
  $G(60, \, 2)$ & $\mathsf{D}_{4, \, 5, \, -1}	 \times \mathbb{Z}_3$ &
  $\langle x, \, y \; | \; x^4=y^5=1, \, xyx^{-1}=y^{-1}  \rangle \times
  \langle z \; | \; z^5=1 \rangle$ \\
&& \\
 $G(60, \, 3)$ & $\mathsf{D}_{4, \, 15, \, -1}$ &
      $\langle x, \, y \; | \;
      x^4=y^{15}=1, \, xyx^{-1}=y^{-1} \rangle$ \\
&& \\
$G(60, \, 5)$ & $\mathsf{A}_5$ &
      $\langle (1 \, 2\, 3\, 4\, 5), \, (3 \, 4\, 5)  \rangle$	\\
&& \\
 $G(60, \, 6)$ & $\mathsf{D}_{4, \, 5, \, -2}	\times \mathbb{Z}_3$ &
 $\langle x, \, y \; | \; x^4=y^5=1, \, xyx^{-1}=y^{-2}  \rangle \times
 \langle z \; | \; z^5=1 \rangle$ \\
&& \\
  $G(60, \, 7)$ & $\mathsf{D}_{4, \, 15, \, -7}$ &
      $\langle x, \, y \; | \;
      x^4=y^{15}=1, \, xyx^{-1}=y^{-7} \rangle$ \\
&& \\
$G(60, \, 8)$ & $\mathsf{S}_3 \times \mathsf{D}_{10}$ & $\langle a, \, b,
\, x, \, y \; | \; a^2=b^3=x^2=y^5=1, $ \\
&& $aba^{-1}=b^{-1}, \, xyx^{-1}=y^{-1}, $ \\
&& $[a, \, x]=1, \, [a, \, y]=1,$ \\
&& $[b, \, x]=1, \, [b, \, y]=1 \rangle$ \\
&& \\
$G(60, \, 9)$ & $\mathsf{A}_4  \times \mathbb{Z}_5$ &$ \langle (1 \, 2)(3 \,
4), \, (1\, 2\, 3) \rangle \times   \langle z \; | \; z^5=1  \rangle $\\
 && \\
$G(60, \, 10)$ & $\mathsf{D}_{10}  \times \mathbb{Z}_6$ & $ \langle x, \,
y  \; | \; x^2=y^5=1, \, xyx^{-1} = y^{-1} \rangle \times   \langle z \; |
\; z^6=1  \rangle $\\
&& \\
 $G(60, \, 11)$ & $\mathsf{S}_3 \times \mathbb{Z}_{10}$ &  $\langle (1 \,
 2), \, (1 \, 2\, 3) \rangle \times \langle z \; | \; z^{10}=1 \rangle$ \\
& &\\
$G(60,\,12)$ & $\mathsf{D}_{60}$ & $\langle x, \, y \; | \; x^2=y^{30}=1,
\, xyx^{-1} = y^{-1}  \rangle$ \\
&& \\
  \bottomrule
	   \end{tabularx}
  \end{center}
 \end{table}










\end{document}